\documentclass[a4paper]{amsart}
\usepackage{amssymb}
\usepackage[all]{xy}

\newtheorem{theorem}{Theorem}[section]
\newtheorem{lemma}[theorem]{Lemma}
\newtheorem{Lemma}[theorem]{Lemma}
\newtheorem{cor}[theorem]{Corollary}

\theoremstyle{definition}
\newtheorem{definition}[theorem]{Definition}

\newtheorem{constr}[theorem]{Construction}

\theoremstyle{remark}

\numberwithin{equation}{section}
\newcommand{\Q}{\mathbb{Q}}
\newcommand{\C}{\mathbb{C}}
\newcommand{\N}{\mathbb{N}}
\newcommand{\R}{\mathbb{R}}

\DeclareMathOperator{\image}{im}

\DeclareMathOperator{\kernel}{ker}

\DeclareMathOperator{\diag}{diag}
\DeclareMathOperator{\codim}{codim}

\title{Circle actions and scalar curvature}

\author{Michael Wiemeler}
\address{Institut f\"ur Mathematik\\Universit\"at Augsburg\\D-86135 Augsburg\\ Germany}
\email{michael.wiemeler@math.uni-augsburg.de}
\thanks{}

\subjclass[2010]{53C20, 57S15}

\keywords{metrics of positive scalar curvature, \(S^1\)-actions}

\begin{document}
\begin{abstract}
  We construct metrics of positive scalar curvature on manifolds with circle actions.
  One of our main results is that there exist \(S^1\)-invariant metrics of positive scalar curvature on every \(S^1\)-manifold which has a fixed point component of codimension \(2\).
  As a consequence we can prove that there are non-invariant metrics of positive scalar curvature on many manifolds with circle actions.
  Results from equivariant bordism allow us to show that there is an invariant metric of positive scalar curvature on the connected sum of two copies of a simply connected semi-free \(S^1\)-manifold \(M\) of dimension at least six provided that \(M\) is not \(\text{spin}\) or that \(M\) is \(\text{spin}\) and the \(S^1\)-action is of odd type.
  If \(M\) is spin and the \(S^1\)-action of even type then there is a \(k>0\) such that the equivariant connected sum of \(2^k\) copies of \(M\) admits an invariant metric of positive scalar curvature if and only if a generalized \(\hat{A}\)-genus of \(M/S^1\) vanishes.
\end{abstract}

\maketitle


\section{Introduction}

In this article we discuss the following questions: Let \(G\) be a compact connected Lie-group and \(M\) a closed connected effective \(G\)-manifold.
\begin{enumerate}
\item Is there a \(G\)-invariant metric of positive scalar curvature on \(M\)?
\item If the answer to the first question is ``no'', does there exist a non-invariant metric of positive scalar curvature on \(M\)? 
\end{enumerate}

It has been shown by Lawson and Yau \cite{MR0358841} that the answer to the first question is ``yes'' if \(G\) is non-abelian.
Therefore we concentrate on the case where \(G\) is abelian and especially on the case \(G=S^1\).

In this case there are two extreme situations:
\begin{enumerate}
\item The \(S^1\)-action on \(M\) is free.
\item There are ``many'' \(S^1\)-fixed points, i.e. the fixed point set has low codimension.
\end{enumerate}

The first situation was studied by B\'erard Bergery \cite{berard83:_scalar}, who showed that a free \(S^1\)-manifold \(M\) admits an \(S^1\)-invariant metric of positive scalar curvature if and only if \(M/S^1\) admits a metric of positive scalar curvature.

For the second case we have the following theorem.

\begin{theorem}[Theorem \ref{sec:constr-invar-psc-1}]
\label{sec:introduction}
  Let \(G\) be a compact Lie-group. Assume that there is a circle
  subgroup \(S^1\subset Z(G)\) contained in the center of \(G\).
  Moreover, let \(M\) be a closed connected effective \(G\)-manifold such that there is a component \(F\) of \(M^{S^1}\) with \(\codim F=2\).
  Then there is an \(G\)-invariant metric of positive scalar curvature on \(M\).
\end{theorem}

A torus manifold \(M\) is a closed connected \(2n\)-dimensional manifold with an effective action of an \(n\)-dimensional torus \(T\), such that \(M^T\neq \emptyset\).
Smooth compact toric varieties are examples of torus manifolds.
As a corollary to Theorem~\ref{sec:introduction} we prove:

\begin{cor}[Corollary \ref{sec:constr-invar-metr-1}]
  Every torus manifold admits an invariant metric of positive scalar curvature.
\end{cor}

We also show that a closed connected semi-free \(S^1\)-manifold \(M\) of dimension greater than five without fixed point components of codimension less than four admits an invariant metric of positive scalar curvature if and only if the bordism class of \(M\) in a certain equivariant bordism group can be represented by a \(S^1\)-manifold with an invariant metric of positive scalar curvature (see Theorem~\ref{sec:semi-free-circle-6}).

If \(M\) is simply connected, one has to distinguish between the following three cases:
\begin{enumerate}
\item \(M\) does not admit a Spin-structure.
\item \(M\) admits a Spin-structure and the \(S^1\)-action lifts into this structure. In this case it is said that the action is of \emph{even type}.
\item \(M\) admits a Spin-structure, but the \(S^1\)-action does not lift into it. In this case it is said that the action is of \emph{odd type}.
\end{enumerate}

An investigation of the relevant bordism groups in these cases leads to the following theorems.

\begin{theorem}[{Theorem \ref{sec:semi-free-circle-1}}]
\label{sec:introduction-2}
  Let \(M\) be a closed simply connected semi-free \(S^1\)-manifold of dimension \(n>5\).
  If \(M\) is not spin or spin and the \(S^1\)-action is odd, then
  the equivariant connected sum of two copies of \(M\) admits an invariant metric of positive scalar curvature.
\end{theorem}

In \cite{MR1758446} Lott constructed a generalized \(\hat{A}\)-genus for orbit spaces of semi-free even \(S^1\)-actions on \(\text{Spin}\)-manifolds.
He showed that for such a manifold \(M\), \(\hat{A}(M/S^1)\) vanishes if \(M\) admits an invariant metric of positive scalar curvature.
We prove the following partial converse to his result.

\begin{theorem}[Theorem \ref{sec:spin-case-1}]
\label{sec:introduction-3}
  Let \(M\) be a closed simply connected Spin-manifold of dimension \(n>5\)  with even semi-free \(S^1\)-action.
  Then we have \(\hat{A}(M/S^1)=0\) if and only if
  there is a \(k\in \N\) such that
   the equivariant connected sum of \(2^k\) copies of \(M\) admits an invariant metric of positive scalar curvature.
\end{theorem}

By using Theorem~\ref{sec:introduction}, we can show that there are non-invariant metrics of positive scalar curvature on many manifolds with \(S^1\)-action.
This is the content of the next theorem.

\begin{theorem}[Theorem \ref{sec:constr-metr-posit-1}]
\label{sec:introduction-1}
  Let \(M\) be a closed connected effective \(S^1\)-manifold of dimension \(n\geq 5\) such that the principal orbits in \(M\) are null-homotopic. 
This condition guarantees that the \(S^1\)-action on \(M\) lifts to an \(S^1\)-action on the universal cover  \(\tilde{M}\) of \(M\).
    If the universal cover of \(M\) is a Spin-manifold assume that the lifted \(S^1\)-action on \(\tilde{M}\) is of odd type.

    Then \(M\) admits a non-invariant metric of positive scalar curvature.
\end{theorem}

Note that if in the situation of the above theorem \(M^{S^1}\neq \emptyset\), then the principal orbits are always null-homotopic.
Since a \(S^1\)-manifold \(M\) with \(\chi(M)\neq 0\) has fixed points we get the following corollary.

\begin{cor}[Corollary \ref{sec:constr-non-invar-2}]
  Let \(M\) be a closed connected manifold of dimension \(n\geq 5\) with non-zero Euler-characteristic such that the universal cover of \(M\) is not spin.
  If \(M\) does not admit a metric of positive scalar curvature, then there is no non-trivial \(S^1\)-action on \(M\).
\end{cor}

It should be noted that the condition on the principal orbits in the above theorem cannot be omitted.
This can be seen by considering a torus \(T^n=\prod_{i=1}^n S^1\) on which \(S^1\) acts by multiplication on one of the factors.
Then \(T^n\) admits a Spin-structure for which the \(S^1\)-action is of odd type.
But \(T^n\) does not admit a metric of positive scalar curvature.

Bredon \cite{MR0221518}, Schultz \cite{MR0380853} and Joseph \cite{MR632190} constructed \(S^1\)-actions of even type on homotopy spheres not bounding Spin-manifolds.
It is known that these homotopy spheres do not admit metrics of positive scalar curvature. 
Therefore Theorem \ref{sec:introduction-1} is not true for \(S^1\)-actions of even type on Spin-manifolds.

This paper is organized as follows.
In Section \ref{sec:constr-invar-psc-3} we prove Theorem~\ref{sec:introduction} and give some applications.
Then in Section \ref{sec:constr-psc-metr-1} we prove Theorem~\ref{sec:introduction-1} and give more applications.
In Section~\ref{sec:semi-free} we discuss the existence of metrics of positive scalar curvature on semi-free \(S^1\)-manifolds without fixed point components of codimension two.

I would like to thank Bernhard Hanke for helpful discussions on the subject of this paper and a simplification of the proof of Theorem~\ref{sec:constr-metr-posit-1}.
Moreover, I would like to thank the Max Planck Institute for Mathematics in Bonn  for hospitality and financial support while I was working on this paper.
I also want to thank the anonymous referee for comments which helped to improve the paper.

\section{Construction of invariant metrics of positive scalar curvature}
\label{sec:constr-invar-psc-3}

In this section we construct invariant metrics of positive scalar curvature on \(S^1\)-manifolds \(M\) such that \(M^{S^1}\) has codimension two.

For the construction of our metrics of positive scalar curvature we use a surgery-principle which was first proven independently by Gromov and Lawson \cite{MR577131} and Schoen and Yau \cite{MR535700}.
Later it was noted by B\'erard Bergery \cite{berard83:_scalar} that these constructions also work in the equivariant setting.
This gives the following theorem.

\begin{theorem}[{\cite[Theorem 11.1]{berard83:_scalar}}]
\label{sec:pre}
  Let \(G\) be a compact Lie-group and \(M\) and \(N\) be \(G\)-manifolds.
  Assume that \(N\) admits an \(G\)-invariant metric of positive scalar curvature.
  If \(M\) is obtained from \(N\) by equivariant surgery of codimension at least three, then \(M\) admits an invariant metric of positive scalar curvature.
\end{theorem}

Besides the construction of metrics of positive scalar curvature via surgery, we need the following result which tells us that there are such metrics on certain orbit spaces of free torus actions.

\begin{theorem}
\label{sec:pre-1}
  Let \(M\) be a manifold with a free action of a torus \(T\). Assume that there is an action of a compact Lie-group which commutes with the \(T\)-action on \(M\).
Then there is a \(G\)-invariant metric of positive scalar curvature on \(M/T\) if and only if there is a \(G\times T\)-invariant metric of positive scalar curvature on \(M\).
\end{theorem}
\begin{proof}
  In the case that \(G\) is the trivial group this theorem is part of B\'erard Bergery's Theorem C from \cite{berard83:_scalar}.
  So we begin by recalling B\'erard Bergery's construction and then indicate what has to be done to get a \(G\)-invariant metric.

  B\'erard Bergery starts with a \(T\)-invariant metric \(g\) of positive scalar curvature on \(M\).
  This metric induces a metric \(g^*\) on \(M/T\), such that the orbit map \(\pi:M\rightarrow M/T\)  is  a Riemannian submersion.
  The metric \(g^*\) is then rescaled by the function \(F:M/T\rightarrow \R\), \(F(x)=f(x)^{2/(\dim M/T-1)}\), where \(f(x)\) is the volume of the \(T\)-orbit \(\pi^{-1}(x)\).
  B\'erard Bergery proves that the resulting metric \(\tilde{g}\) has positive scalar curvature.

  If we choose the metric \(g\) to be \(G\times T\)-invariant, then 
  every element of \(G\) maps each \(T\)-orbit in \(M\) isometrically onto another \(T\)-orbit because \(G\) and \(T\) commute.
 Therefore the function \(F\) is \(G\)-invariant.

 Moreover, if \(h\) is an element of \(G\) and \(x\in M\), then the differential \(D_xh\) maps the orthogonal complement of \(T_x(Tx)\) in \(T_xM\) isometrically onto the orthogonal complement of \(T_{hx}(Thx)\) in \(T_{hx}M\).
 Since \(\pi:(M,g)\rightarrow (M/T,g^*)\) is a Riemannian submersion, it follows that \(g^*\) is \(G\)-invariant.

  Hence, the metric \(\tilde{g}\) is also \(G\)-invariant.

  For the construction of an invariant metric with positive scalar curvature on \(M\) from a metric on \(M/T\), just pick a \(G\times T\)-invariant connection for the principal \(T\)-bundle \(M\rightarrow M/T\) and a flat invariant metric \(h\) on \(T\).
Then by a result of Vilms \cite{MR0262984} there is a unique metric on \(M\) such that \(M\rightarrow M/T\) is a Riemannian submersion with totally geodesic fibers isometric to \((T,h)\) and horizontal distribution associated to the chosen connection.
By construction this metric is \(G\times T\)-invariant.
  After shrinking the fibers one gets a metric of positive scalar curvature on \(M\).
In the following we will call a metric obtained by such a construction a connection metric.
\end{proof}

Now we turn to our construction of invariant metrics of positive scalar curvature.
For this we need the following lemma.

\begin{lemma}
\label{sec:constr-invar-psc}
  Let \(G\) be a compact Lie-group and \(Z\) be a compact connected \(G\)-manifold with non-empty boundary.
  We view \(Z\) as a bordism between the empty set and \(\partial Z\).
  Then there is a \(G\)-handle decomposition of \(Z\) without handles of codimension \(0\). 
\end{lemma}
\begin{proof}
  We choose a special \(G\)-Morse function \(f:Z\rightarrow [0,1]\) without critical orbits on the boundary of \(Z\), such that
  \(f^{-1}(1)=\partial Z\).
For the definition of special \(G\)-Morse functions and some of their properties see \cite{MR2376283} or  \cite{MR975472}.
The map \(f\) induces a \(G\)-handle decomposition of \(Z\) such that the handles correspond one-to-one to the critical orbits of \(f\).
The codimension of a handle corresponding to a critical orbit is given by the coindex of this orbit.
Therefore we have to show that we can change \(f\) in such a way that there  are no critical orbits of coindex \(0\).
By Lemma 13 of \cite{MR2376283}, the critical orbits of coindex \(0\) are principal orbits.

Therefore, as in the proof of Theorem 15 of \cite{MR2376283}, non-equivariant handle cancellation on the orbit space can be used to remove all handles of codimension \(0\).
Note here that we only work with handles of codimension \(0\).
Therefore we do not need the dimension assumption from Hanke's theorem.
\end{proof}

Now we can prove the first of our main theorems.

\begin{theorem}
\label{sec:constr-invar-psc-1}
  Let \(G\) be a compact Lie-group. Assume that there is a circle
  subgroup \(S^1_0\subset Z(G)\) contained in the center of \(G\).
  Moreover, let \(M\) be a closed connected effective \(G\)-manifold such that there is a component \(F\) of \(M^{S^1_0}\) with \(\codim F=2\).
  Then there is an \(G\)-invariant metric of positive scalar curvature on \(M\).
\end{theorem}
\begin{proof}
  Let \(Z\) be \(M\) with an open tubular neighborhood of \(F\) removed.
  Then \(Z\) is a \(G\)-manifold with boundary SF, the normal sphere bundle of \(F\).
  Let \(S^1_1\) act on \(D^2\subset\C\) via multiplication with the inverse.
  Then we have a \(G\times S^1_1\)-manifold \(X=Z\times D^2\) (with corners equivariantly smoothed).
  A \(G\)-handle decomposition of \(Z\) induces a \(G\times S^1_1\)-handle decomposition of \(X\) (both viewed as bordisms between the empty set and their boundaries) such that
  \begin{enumerate}
  \item the \(G\)-handles of \(Z\) are one-to-one to the \(G\times S^1_1\)-handles of \(X\).
  \item the codimension of a handle of \(X\) is given by the codimension of the corresponding handle of \(Z\) plus two.
  \end{enumerate}
By Lemma~\ref{sec:constr-invar-psc}, we may assume that there is no handle of codimension \(0\) in the decomposition of \(Z\).
Therefore in the decomposition of \(X\) all handles have codimension at least three.
Hence, it follows from Theorem~\ref{sec:pre} that \(\partial X\) admits an \(G\times S^1_1\)-invariant metric of positive scalar curvature.

Now we have
\begin{equation*}
  \partial X = SF\times D^2 \cup_F Z\times S^1,
\end{equation*}
where the gluing map \(F:\partial(SF\times D^2)=SF\times S^1\rightarrow
SF\times S^1=\partial(Z\times S^1)\) is given by \(f\times g^{-1}\).
Here \(f:SF\hookrightarrow Z\) and \(g:S^1\hookrightarrow D^2\) are
the natural inclusions.

Note that \(SF\) is a principal \(S^1\)-bundle over \(F\) and the action of \(S^1_0\) on this bundle is given by multiplication on the fibers.
Therefore the orbit space of the free \(\diag(S^1_0\times S^1_1)\)-action on \(SF\times D^2\) is the normal disc bundle of \(F\) in \(M\).
A diffeomorphism is induced by the map
\begin{align*}
  SF\times D^2&\rightarrow N(F,M)&(x,\lambda)&\mapsto \lambda x,
\end{align*}
where \(\lambda x\) is given by complex multiplication of \(\lambda\in
D^2\subset \C\) and \(x\in SF\) with respect to the complex structure on \(N(F,M)\)
induced by the action of \(S^1_0\).

Moreover, the orbit space of the free \(\diag(S^1_0\times
S^1_1)\)-action on \(Z\times S^1\) is diffeomorphic to \(Z\).
A diffeomorphism is induced by the map
\begin{align*}
  Z\times S^1&\rightarrow Z&(x,\lambda)&\mapsto \lambda\cdot x,
\end{align*}
where \(\cdot\) denotes the action of \(S^1_0\) on \(Z\).

Hence, it follows from the special form of \(F\) described above that \(F\)
induces the natural inclusion \(SF\hookrightarrow Z\) on the orbit space.

Therefore the quotient of the free \(\diag(S^1_0\times S^1_1)\)-action on \(\partial X\) is \(G\)-equivariantly  diffeomorphic to \(M\).
Since the action of \(G\) on \(\partial X\) commutes with the action of \(\diag(S^1_0\times S^1_1)\), it follows from Theorem~\ref{sec:pre-1} that \(M\) admits a \(G\)-invariant metric of positive scalar curvature. 
\end{proof}

Note that if a torus \(T\) acts effectively on a manifold \(M\), then all fixed point components have codimension greater or equal to \(2\dim T\).
For the case of equality we have the following corollary. 

\begin{cor}
\label{sec:constr-invar-psc-2}
  Let \(T\) be a torus that acts effectively on the closed connected manifold \(M\).
  If there is a component of \(M^T\) of codimension \(2\dim T\),
  then there is an \(T\)-invariant metric of positive scalar curvature on \(M\).
\end{cor}
\begin{proof}
  Let \(F\) be a component of \(M^T\) of codimension \(2\dim T\) and \(x\in F\).
  Then up to an automorphism of \(T\) the \(T\)-representation on the normal space \(N_x(F,M)\) is given by the standard \(T\)-representation.
  Therefore there is a codimension-two subspace of \(N_x(F,M)\) which is fixed pointwise by a circle subgroup \(S^1\subset T\).
  Hence, it follows that \(F\) is contained in a codimension-two submanifold of \(M\) which is fixed pointwise by \(S^1\).
  Now the statement follows from Theorem~\ref{sec:constr-invar-psc-1}.
\end{proof}

A torus manifold is a closed connected \(2n\)-dimensional manifold \(M\) with an effective action of an \(n\)-dimensional torus \(T\), such that \(M^T\neq \emptyset\).
Smooth compact toric varieties are examples of torus manifolds.
As a consequence of Corollary~\ref{sec:constr-invar-psc-2} we get the following corollary.

\begin{cor}
\label{sec:constr-invar-metr-1}
  Every torus manifold admits an invariant metric of positive scalar curvature.
\end{cor}
\begin{proof}
  The fixed point set of a torus manifold consists of isolated points. Therefore the statement follows from Corollary \ref{sec:constr-invar-psc-2}.
\end{proof}

As an application of Corollary~\ref{sec:constr-invar-metr-1} we can improve the upper bound for the degree of symmetry of manifolds which do not admit metrics of positive scalar curvature given by Lawson and Yau in \cite{MR0358841}.

\begin{cor}
\label{sec:constr-invar-metr}
  Let \(G\) be a compact connected Lie-group which acts effectively on a closed connected  manifold \(M\) which does not admit a metric of positive scalar curvature and has non-zero Euler-characteristic.
  Then \(G\) is a torus of dimension less than \(\frac{1}{2}\dim M\).
\end{cor}
\begin{proof}
  By a result of Lawson and Yau \cite{MR0358841}, \(G\) must be a torus.
  Because the Euler-characteristic of \(M\) is non-zero, there are \(G\)-fixed points in \(M\).
  Therefore we have \(2\dim G\leq \dim M\).
  If equality holds in this inequality, then \(M\) is a torus manifold.
  Therefore, by Corollary \ref{sec:constr-invar-metr-1}, we have \(2\dim G<\dim M\).
\end{proof}

 By \cite[Corollary 2.7]{MR0219077}, a homotopy sphere of dimension
 greater than one bounds a
 Spin-manifold if and only if it has vanishing \(\alpha\)-invariant.
 Moreover, homotopy spheres with non-vanishing \(\alpha\)-invariant
 only exist in dimensions greater or equal to \(9\) congruent to \(1\)
 or \(2 \mod 8\).
 In these dimensions they constitute half of all homotopy spheres (see
 \cite[Proof of Theorem 2]{MR0180978} and \cite[Theorem 1.2]{MR0198470}).
 Since the \(\alpha\)-invariant vanishes for spin-manifolds which
 admit metrics of positive scalar curvature, such homotopy spheres do not 
 admit a metric of positive scalar curvature.

Hence, even-dimensional exotic spheres \(\Sigma\)  which do not bound Spin-manifolds are examples of manifolds for which the assumptions on \(M\) from the above corollary hold.
Other examples of manifolds \(M\) can be constructed as follows. 
Let \(N\) be a torus manifold which is spin with \(10\leq \dim N\equiv 2\mod 8\).
Then the \(\alpha\)-invariant of \(M=N\# \Sigma\) does not vanish.
Therefore \(M\) does not admit a metric of positive scalar curvature and
 satisfies the assumptions of the corollary.
Therefore \(M\) does not admit a smooth action of a torus of dimension \(\frac{1}{2}\dim M=\frac{1}{2}\dim N\).

For four-dimensional \(S^1\)-manifolds we also have the following two corollaries.

\begin{cor}
  Let \(M\) be a closed connected effective \(S^1\)-manifold with \(\dim M=4\) and \(\chi(M)<0\).
  Then there is an invariant metric of positive scalar curvature on \(M\).
\end{cor}
\begin{proof}
  Since \(\chi(M^{S^1})=\chi(M)<0\), there must be a fixed point component of dimension two.
 Hence, the corollary follows from Theorem~\ref{sec:constr-invar-psc-1}.
\end{proof}

\begin{cor}
  Let \(M\) be a closed connected oriented semi-free \(S^1\)-manifold with \(\dim M=4\) and non-vanishing signature. Then there is an invariant metric of positive scalar curvature on \(M\).
\end{cor}
\begin{proof}
  By Corollary 6.24 of \cite{MR1150492}, there is a fixed point component of dimension two in \(M\).
 Hence, the corollary follows from Theorem~\ref{sec:constr-invar-psc-1}.
\end{proof}

\section{Constructions of non-invariant metrics of positive scalar curvature}
\label{sec:constr-psc-metr-1}

In this section we construct non-invariant metrics of positive scalar curvature on many manifolds with circle actions. 
For this construction we need the following lemma.

\begin{lemma}
\label{sec:constr-psc-metr}
  Let \(M\) be a closed connected effective \(S^1\)-manifold of dimension \(n>1\).
  If \(n>2\), then \(M\) is equivariantly bordant to a closed connected effective \(S^1\)-manifold \(N\) such that there is a component of \(N^{S^1}\) of codimension two.
  If \(n=2\), then \(M\) is equivariantly bordant to a closed effective \(S^1\)-manifold with at most two components such that each component contains an isolated \(S^1\)-fixed point.
  Moreover, we have:
  \begin{enumerate}
  \item \label{item:1} If \(M\) is oriented, then \(N\) is also oriented and they are equivariantly bordant as oriented manifolds.
  \item \label{item:3} If \(M\) is spin and the \(S^1\)-action on \(M\) is of odd type, then the same holds for \(N\) and they are equivariantly bordant as Spin-manifolds.
   \end{enumerate}
\end{lemma}
\begin{proof}
  Let \(S^1_0\hookrightarrow M\) be the inclusion of a principal orbit in \(M\).
  Then the equivariant normal bundle of \(S_0^1\) is given by \(S_0^1\times \R^{n-1}\), where \(S^1\) acts trivially on the \(\R^{n-1}\)-factor.
  Therefore we can do equivariant surgery on \(S^1_0\) to obtain an \(S^1\)-manifold \(N\).
  If \(n>2\), then \(N\) is connected and has a fixed point component of codimension two.
  If \(n=2\), then \(N\) might have two components both containing an isolated \(S^1\)-fixed point.
  If we let \(W\) be the trace of this surgery we see that (\ref{item:1}) holds.

  For (\ref{item:3}) we need an extra argument to show that \(W\) is spin.
  This is the case if the inclusions
  \begin{align*}
    S_0^1\times D^{n-1}&\hookrightarrow M&&\text{and}& S_0^1\times D^{n-1}&\hookrightarrow D^2\times D^{n-1}
  \end{align*}
  induce the same Spin-structure on \(S_0^1\times D^{n-1}\).

  Let \(\hat{S}^1\) be the connected double cover of \(S^1\).
  Then the \(S^1\)-action on \(M\) induces an action of \(\hat{S}^1\) on \(M\).
  This action of \(\hat{S}^1\) lifts into the Spin-structure on \(M\).
  Let \(\mathbb{Z}_2\subset \hat{S}^1\) be the kernel of \(\hat{S}^1\rightarrow S^1\).
  If the \(S^1\)-action on \(M\) is of odd type, then \(\mathbb{Z}_2\) acts non-trivially on each fiber of the principal Spin-bundle over \(M\).
  If the \(S^1\)-action on \(M\) is of even type, then \(\mathbb{Z}_2\) acts trivially on this bundle.

  On \(S_0^1\times D^{n-1}\) there are two Spin-structures.
  For one of them the \(S^1\)-action on \(S_0^1\times D^{n-1}\) is of even type.
  For the other it is of odd type.
  By the above remark about the action of \(\mathbb{Z}_2\), it follows that the inclusion \(S_0^1\times D^{n-1}\hookrightarrow M\), induces the second Spin-structure on \(S_0^1\times D^{n-1}\).
  This is also the Spin-structure which is induced by the inclusion of \(S_0^1\times D^{n-1}\hookrightarrow D^2\times D^{n-1}\).
  Therefore \(W\) is an equivariant Spin-cobordism.
\end{proof}

By combining Theorem \ref{sec:constr-invar-psc-1} and Lemma \ref{sec:constr-psc-metr} we recover the following result of Ono.

\begin{cor}[\cite{MR1043417}]
  Let \(M\) be a closed connected Spin-manifold with a \(S^1\)-action of odd type. Then the \(\alpha\)-invariant of \(M\)  vanishes.
\end{cor}
\begin{proof}
  The \(\alpha\)-invariant is a Spin-bordism invariant and vanishes for Spin-manifolds which admit metrics of positive scalar curvature.
  Since an action of odd type is always non-trivial, we may assume that the \(S^1\)-action on \(M\) is effective.
  Because by Theorem \ref{sec:constr-invar-psc-1} and Lemma \ref{sec:constr-psc-metr}, \(M\) is Spin-bordant to a manifold with a metric of positive scalar curvature the statement follows.
\end{proof}

As another application of Theorem~\ref{sec:constr-invar-psc-1} we get:

\begin{theorem}
\label{sec:constr-metr-posit-1}
  Let \(M\) be a closed connected effective \(S^1\)-manifold of dimension \(n\geq 5\) such that the principal orbits in \(M\) are null-homotopic. 
This condition guarantees that the \(S^1\)-action on \(M\) lifts to an \(S^1\)-action on the universal cover \(\tilde{M}\) of \(M\).
    If the universal cover of \(M\) is a Spin-manifold assume that the lifted \(S^1\)-action on \(\tilde{M}\) is of odd type.

    Then \(M\) admits a non-invariant metric of positive scalar curvature.
\end{theorem}
\begin{proof}
  At first do an equivariant surgery on a principal orbit of the \(S^1\)-action on \(M\) as in the proof of Lemma~\ref{sec:constr-psc-metr}.
  By Theorem~\ref{sec:constr-invar-psc-1}, the resulting manifold has an invariant metric of positive scalar curvature.
  Moreover, since the principal orbits are null-homotopic, it follows from the assumption on the lifted action to the universal cover, that \(N\) is diffeomorphic to the connected sum of \(M\) and \(S^2\times S^{n-2}\).

  Indeed, the result of a surgery on a null-homotopic circle \(S\) in \(M\) is diffeomorphic to the connected sum of \(M\) and \(S^2\times S^{n-2}\) or the connected sum of \(M\) and the non-trivial \(S^{n-2}\)-bundle over \(S^2\), depending on the choice of the framing of the circle, for which there are two choices.
  If the universal cover of \(M\) is not spin, then these two manifolds are diffeomorphic.
  To see this take an embedding of \(S^2\hookrightarrow M\) with non-trivial normal bundle.
  We may assume that \(S\) is contained in \(S^2\) as a small circle around the north pole.
  We fix a framing of \(S\) in \(M\).
  We may move \(S\) along the meridians of \(S^2\) to the south pole and then rotate \(S^2\) so that the north and south pole are interchanged.
  During the way of \(S\) in \(S^2\) its framing changes, so that the result of a surgery on \(S\) is independent of the choice of the framing.

  If the universal cover of \(M\) is spin, then the assumption on the lifted action and the argument from the proof of Lemma~\ref{sec:constr-psc-metr} imply that the universal cover of the result of the surgery is also spin.
  Therefore \(N\) is the connected sum of \(M\) and \(S^2\times S^{n-2}\).
  Hence, by surgery on the \(S^2\)-factor we recover \(M\).
  Since this surgery is of codimension at least three, it follows that \(M\) admits a metric of positive scalar curvature.
\end{proof}

We give an example which shows that the assumptions of the above theorem are not sufficient for the existence of \(S^1\)-invariant metrics.
Let \(X=(\C P^2\times \C P^1)\#T^6\) and M be the principal \(S^1\)-bundle over \(X\) with first Chern class a generator of \(H^2(\C P^2;\mathbb{Z})\subset H^2(X;\mathbb{Z})\).
Then it follows from an inspection of a Mayer-Vietoris-sequence that \(M\) is a Spin-manifold.
Moreover, the \(S^1\)-action on M is of odd type because \(X\) is not a Spin-manifold.
The \(S^1\)-orbits in \(M\) are null-homotopic because \(M|_{\C P^2\subset X}=S^5\) is simply connected.
But there is a degree-one map \(X\rightarrow T^6\).
Hence, \(X\) does not admit a metric of positive scalar curvature by
\cite[Corollary 2]{MR535700}.
Therefore, by Theorem~\ref{sec:pre-1}, \(M\) does not admit a \(S^1\)-invariant metric of positive scalar curvature.

As a consequence of the proof of Theorem~\ref{sec:constr-metr-posit-1} we get the following corollary.

\begin{cor}
  Let \(M\) be a manifold of dimension \(n\geq 4\) with effective \(S^1\)-action and null-homotopic principal orbits. Denote by \(N\) the non-trivial \(S^{n-2}\)-bundle over \(S^2\). Then the following holds:
  \begin{enumerate}
  \item One of the manifolds \(M\# (S^2\times S^{n-2})\) or \(M\# N\) admits an \(S^1\)-action with an invariant metric of positive scalar curvature.
  \item \(M\# N \# (S^2\times S^{n-2})\) admits an \(S^1\)-action with an invariant metric of positive scalar curvature.
  \end{enumerate}
\end{cor}
\begin{proof}
  Let \(M'\) be the result of an equivariant surgery on a principal orbit in \(M\).
  Then the \(S^1\)-action on \(M'\) has a fixed point component of codimension two.
  Therefore \(M'\) has an invariant metric of positive scalar curvature.
  Since the principal orbits in \(M\) are null-homotopic, \(M'\) is diffeomorphic to \(M\# (S^2\times S^{n-2})\) or \(M\# N\). Therefore the first claim follows.

  The structure group of \(N\rightarrow S^2\) reduces to \(S^1\).
  Let \(S^1\) act on \(S^2\) by rotation.
  Then, by results of Hattori and Yoshida \cite{MR0461538}, the \(S^1\)-action on \(S^2\) lifts to an action on \(N\) such that the action on the fiber over one of the fixed points in \(S^2\) is trivial.
  Therefore the lifted action on \(N\) has a fixed point component of codimension two.
  A similar statement holds for an \(S^1\)-action on \(S^2\times S^{n-2}\).

  Then we can form the connected sum of \(M'\) and \(N\) (and also of \(M'\) and \(S^2\times S^{n-2}\)) in an equivariant way.
  Therefore \(M'\# N\) and \(M'\# (S^2\times S^{n-2})\) admit an \(S^1\)-action with an invariant metrics of positive scalar curvature. This proves the second claim.
\end{proof}

Since the principal orbits of an \(S^1\)-action are always null-homotopic if the \(S^1\)-action has fixed points,
we get the following corollary.

\begin{cor}
\label{sec:constr-non-invar-2}
  Let \(M\) be a closed connected manifold of dimension \(n\geq 5\) with non-zero Euler-characteristic such that the universal cover of \(M\) is not spin.
  If \(M\) does not admit a metric of positive scalar curvature, then there is no non-trivial \(S^1\)-action on \(M\).
\end{cor}
\begin{proof}
  Since \(\chi(M)\neq 0\) every \(S^1\)-action on \(M\) must have fixed points.
  Therefore it follows from Theorem~\ref{sec:constr-metr-posit-1} that there is no non-trivial \(S^1\)-action on \(M\).
\end{proof}

It follows from Corollary~\ref{sec:constr-non-invar-2}, that the manifold \(X\) from the example after Theorem~\ref{sec:constr-metr-posit-1} does not admit any non-trivial circle action.

It was an idea of Bernhard Hanke to combine Corollary~\ref{sec:constr-non-invar-2} with ideas of Schick to construct new obstructions to \(S^1\)-actions on manifolds with non-spin universal cover of dimension greater than four.
In the remainder of this section we describe what grew out of this idea.

For these manifolds, there is only one known obstruction to a metric of positive scalar curvature, namely the minimal hypersurface method of Schoen and Yau \cite{MR535700}.
Using this method the following theorem was proved by Joachim and Schick.

\begin{theorem}[{\cite{MR1778107}}]
\label{sec:constr-non-invar-1}
  Let \(G\) be a discrete group and \([M\rightarrow BG]\in \Omega_n^{SO}(BG)\) with  \(2\leq n\leq 8\) and \(M\) connected.
  If there are \(\alpha_1,\dots,\alpha_{n-2}\in H^1(BG;\mathbb{Z})\) such that
  \begin{equation*}
    \alpha_1\cap(\dots\cap(\alpha_{n-2}\cap[M])\dots)\neq 0\in\Omega_2^{SO}(BG),
  \end{equation*}
then \(M\) does not admit a metric of positive scalar curvature.
\end{theorem}
Here, the map
\begin{equation*}
  \cap:H^1(BG;\mathbb{Z})\times \Omega^{SO}_n(BG)\rightarrow \Omega^{SO}_{n-1}(BG),\quad (\alpha,[M])\mapsto \alpha\cap[M]
\end{equation*}
is defined as follows. Represent an element \(\alpha\in H^1(BG;\mathbb{Z})\) by a map \(f:BG\rightarrow S^1\) and let \([\phi:M\rightarrow BG]\in \Omega^{SO}_n(BG)\).
Let \(\psi:M\rightarrow S^1\) be a differentiable map homotopic to \(f\circ \phi\) and \(x \in S^1\) a regular value of \(\psi\).
Then \(\alpha\cap [M]\) is represented by the restriction of \(\phi\) to the hypersurface \(\psi^{-1}(x)\subset M\).

So, by combining Theorem~\ref{sec:constr-metr-posit-1} and Theorem~\ref{sec:constr-non-invar-1}, one gets the following Theorem \ref{sec:constr-non-invar} for certain manifolds of dimension greater than four and smaller than nine.
But there is also a direct proof for this theorem which does not use the scalar curvature of metrics on the manifolds involved.
We give this proof below.

\begin{theorem}
\label{sec:constr-non-invar}
  Let \(G\) be a discrete group and \([\phi: M\rightarrow BG]\in \Omega_n^{SO}(BG)\) with  \(n\geq 2\) and \(M\) connected.
  If there are \(\alpha_1,\dots,\alpha_{n-2}\in H^1(BG;\mathbb{Z})\) such that
  \begin{equation*}
    \alpha_1\cap(\dots\cap(\alpha_{n-2}\cap[M])\dots)\neq 0\in\Omega_2^{SO}(BG),
  \end{equation*}
then \(M\) does not admit an effective \(S^1\)-action with null-homotopic principal orbits.
\end{theorem}
\begin{proof}
  Assume that there is such an action on \(M\).
  At first note that the map \(\psi\) from the above construction can be assumed to be \(S^1\)-equivariant with \(S^1\) acting trivially on \(S^1\).
  This is because, by \cite[Theorem 1.11]{MR670746},
\(\pi_1(M/S^1)=\pi_1(M)/H\) where \(H\), is generated by elements of finite order.
Therefore up to homotopy every map \(M\rightarrow S^1\) factors through \(M/S^1\).
Note also, that if \(N\subset M\) is an invariant submanifold, then the restriction of \(f\circ \phi\) to \(N\) factors up to homotopy through \(M/S^1\) and therefore also through \(N/S^1\).

Hence, by applying the construction \(\cap\) described above several times, we get an orientable invariant two-dimensional submanifold \(N'\) of \(M\) with equivariantly trivial normal bundle.
Therefore the \(S^1\)-action on \(N'\) is effective.
From the classification of  orientable \(S^1\)-manifolds with one-dimensional orbit space it follows that \(N'\) is equivariantly diffeomorphic to a sphere with \(S^1\) acting by rotation or to a torus \(S^1\times S^1\) with \(S^1\) acting by multiplication on the first factor.
In the first case \([N']=0\in\Omega^{SO}_2(BG)\) because \(\pi_2(BG)=0\).
In the second case the map \(N'=S^1\times S^1\rightarrow BG\) extends to a map \(D^2\times S^1\rightarrow BG\) because the principal orbits in \(M\) are null-homotopic.
Hence, \(N'\) is a boundary.
Therefore the statement follows.
\end{proof}

It follows from the above theorem that Schick's five-dimensional counterexample \cite{MR1632971} to the Gromov--Lawson--Rosenberg Conjecture does not admit any \(S^1\)-action with fixed points.
If one does Schick's construction in dimension six, then one gets a manifold with non-trivial Euler characteristic.
Therefore this manifold does not admit any \(S^1\)-action.

\section{Semi-free circle actions}
\label{sec:semi-free}

In this section we discuss the question when a semi-free \(S^1\)-manifold without fixed point components of codimension less than four admits an invariant metric of positive scalar curvature.
For this discussion we need some notations and results of Stolz \cite{stolz:_concor}.

A supergroup \(\gamma\) is a triple \((\pi,w,\hat{\pi})\),  where \(\pi\) is a group, \(\hat{\pi}\rightarrow \pi\) is an extension of \(\pi\) such that \(\kernel (\hat{\pi}\rightarrow \pi)\) has order less than three and \(w:\pi\rightarrow \mathbb{Z}_2\) is a homomorphism.
We call a supergroup discrete if \(\pi\) is a discrete group.

For a discrete supergroup \(\gamma\) one defines a Lie group \(G(n,\gamma)\) as the even part of the superproduct \(\text{Pin}(n)\hat{\times}\gamma\).
There is a homomorphism \(G(n,\gamma)\rightarrow O(n)\), which is surjective if \(w\neq 0\) and has image \(SO(n)\) if \(w=0\).
In both cases \(G(n,\gamma)\) is a covering of its image under this homomorphism.

A \(\gamma\)-structure on a vector bundle \(E\rightarrow X\) equipped with an inner product and with \(\dim E\geq 3\), is a reduction of structure group of \(E\) through the homomorphism  \(G(n,\gamma)\rightarrow O(n)\), i.e. a \(\gamma\)-structure is a principal \(G(n,\gamma)\)-bundle \(P_{G(n,\gamma)}(E)\) over \(X\) together with an isomorphism of principal \(O(n)\)-bundles \(\xi:P_{G(n,\gamma)}(E) \times_{G(n,\gamma)}O(n)\rightarrow P_{O(n)}(E)\), where \(P_{O(n)}(E)\) is the orthogonal frame bundle of \(E\).
If \(M\) is a manifold, then  a \(\gamma\)-structure on \(M\) is a \(\gamma\)-structure on its tangent bundle.
A manifold with a \(\gamma\)-structure is called \(\gamma\)-manifold.

It can be shown that there is a natural bijection between the \(\gamma\)-structures on \(E\) and \(E\oplus \R\).
Moreover, in the case that \(w=0\), a \(\gamma\)-structure induces an orientation on \(E\).

For each vector bundle \(E\rightarrow X\) over a connected space \(X\) there is a supergroup \(\gamma(E)\), which encodes the information contained in the fundamental group of \(X\) and the first two Stiefel-Whitney classes of \(E\).
It is defined as follows:
\begin{enumerate}
\item \(\pi=\pi_1(X)\),
\item \(w:\pi\rightarrow \mathbb{Z}_2\) is the orientation character of \(E\),
\item \(\hat{\pi}\rightarrow \pi\) is the extension of \(\pi \) induced by the projection map \(P_{O(n)}(E)/\langle r\rangle \rightarrow X\), where \(r\in O(n) \) is the reflection in the hyperplane perpendicular to \(e_1=(1,0,\dots,0)\in \R^n\).
\end{enumerate}

A vector bundle together with a choice of a base point for its orthogonal frame bundle is called a pointed vector bundle.
Stolz shows that for every pointed vector bundle \(E\) there is a canonical \(\gamma(E)\)-structure on \(E\).
We denote this \(\gamma(E)\)-structure of \(E\) by \(P_{\gamma(E)}(E)\).

Now let \(X\) be a \(S^1\)-space. Then we call a \(\gamma\)-structure
on an \(S^1\)-vector bundle \(E\rightarrow X\) equivariant if the
\(S^1\)-action on \(P_{O(n)}(E)\) lifts to an \(S^1\)-action on
\(P_{G(n,\gamma)}(E)\) in such a way that the \(S^1\)-action commutes
with the \(G(n,\gamma)\)-action.
If the \(S^1\)-action on \(X\) is free, then there is a natural isomorphism \(E\cong p^*(E/S^1)\), where \(p:X\rightarrow X/S^1\) is the orbit map.
In this case we have the following lemma.

\begin{Lemma}
\label{sec:semi-free-circle-5}
  Let \(E\rightarrow X\) be a \(S^1\)-vector bundle over a free \(S^1\)-space with \(\dim E\geq 3\) and \(\gamma\) a discrete supergroup. Then the following two statements hold:
  \begin{enumerate}
  \item\label{item:4} There is a bijection between  the \(\gamma\)-structures on \(E/S^1\) and the \(S^1\)-equivariant \(\gamma\)-structures on \(E\), induced by pullback.
  \item\label{item:2}
    Equip \(P_{O(n)}(E)\) and  \(P_{O(n)}(E/S^1)=P_{O(n)}(E)/S^1\) with basepoints \(x\) and \(S^1x\).
    Then there are homomorphisms
    \begin{align*}
       \gamma(E)&\rightarrow \gamma(E/S^1)& &\text{and}& G(n,\gamma(E))&\rightarrow G(n,\gamma(E/S^1))
    \end{align*}
    and an isomorphism of \(\gamma(E/S^1)\)-structures
    \begin{equation*}
P_{\gamma(E)}(E)\times_{G(n,\gamma(E))} G(n,\gamma(E/S^1)) \rightarrow p^*P_{\gamma(E/S^1)}(E/S^1).      
    \end{equation*}
  \end{enumerate}
\end{Lemma}
\begin{proof}
  At first we prove (\ref{item:4}).
  If \(P_\gamma(E)\) is an equivariant \(\gamma\)-structure on \(E\), then the isomorphism \(P_{\gamma}(E)\times_{G(n,\gamma)}O(n)\rightarrow P_{O(n)}(E)\) induces an isomorphism
  \begin{equation*}
     P_{\gamma}(E)/S^1\times_{G(n,\gamma)}O(n)\rightarrow P_{O(n)}(E)/S^1=P_{O(n)}(E/S^1). 
  \end{equation*}
  Hence, \(P_{\gamma}(E)/S^1\) is a \(\gamma\)-structure for \(E/S^1\).

  Otherwise, if \(P_\gamma(E/S^1)\) is a \(\gamma\)-structure for \(E/S^1\), then the isomorphism 
  \begin{equation*}
    P_{\gamma}(E/S^1)\times_{G(n,\gamma)}O(n)\rightarrow P_{O(n)}(E/S^1),
  \end{equation*}
 induces an \(S^1\)-equivariant  isomorphism  \(p^*P_{\gamma}(E/S^1)\times_{G(n,\gamma)}O(n)\rightarrow p^*P_{O(n)}(E/S^1)\).
  Therefore \(p^*P_{\gamma}(E/S^1)\) is an equivariant \(\gamma\)-structure for \(p^*(E/S^1)\).

  Since, for an equivariant \(\gamma\)-structure \(P_\gamma(E)\), \(p^*(P_{\gamma}(E)/S^1)\) is naturally isomorphic to \(P_\gamma(E)\) and \(p^*(E/S^1)\) is naturally isomorphic to \(E\),
  these two operations are inverse to each other.
  Hence, (\ref{item:4}) is proved.

  Now we prove (\ref{item:2}).
  The existence of the homomorphisms \(\gamma(E)\rightarrow \gamma(E/S^1)\) and \(G(n,\gamma(E))\rightarrow G(n,\gamma(E/S^1))\) follows from the definitions of \(\gamma(E)\) and \(G(n,\gamma(E))\) and the choices of the basepoints.
  Therefore we only have to prove the existence of the isomorphism of the \(\gamma(E/S^1)\)-structures.
  
  Such an isomorphism exists if and only if there is a \(G(n,\gamma(E))\)-equivariant map \(\phi\) such that the following diagram commutes
  \begin{equation*}
    \xymatrix{
      P_{\gamma(E)}(E) \ar[r]\ar_{\phi}[d]&P_{\gamma(E)}(E)\times_{G(n,\gamma(E))}O(n)\ar^(.65){\cong}[r] & P_{O(n)}(E)\\
      p^*P_{\gamma(E/S^1)}(E/S^1)\ar[r]&p^*P_{\gamma(E)}(E/S^1)\times_{G(n,\gamma(E/S^1))}O(n)\ar^(.65){\cong}[r]& p^*P_{O(n)}(E/S^1)\ar_{\cong}[u]} 
  \end{equation*}

Now, by the proof of \cite[Proposition 2.12]{stolz:_concor}, \(P_{\gamma(E)}(E)\) can be identified with
\begin{equation*}
  \{a:[0,1]\rightarrow P_{O(n)}(E);\;a(0)=x\}/\sim.
\end{equation*}
Here two paths are identified if they are homotopic relative endpoints.
Moreover, the map \(P_{\gamma(E)}(E)\rightarrow P_{O(n)}(E)\) is given by \([a]\mapsto a(1)\).

Clearly, this map factors through
\begin{align*}
  \phi:P_{\gamma(E)}(E)&\rightarrow p^*P_{\gamma(E/S^1)}(E/S^1)=\{(y,e)\in X\times E/S^1;\; p(y)=\psi'(e)\}\\
    [a]&\mapsto (\psi(a(1)),[\bar{p}\circ a])
\end{align*}
Here \(\psi:P_{O(n)}(E)\rightarrow X\) and \(\psi':P_{O(n)}(E/S^1)\rightarrow X/S^1\) denote the bundle projections and 
\begin{equation*}
  \bar{p}:P_{O(n)}(E)\rightarrow P_{O(n)}(E)/S^1=P_{O(n)}(E/S^1)
\end{equation*}
 is the orbit map.
 It follows from the description of the \(G(n,\gamma(E))\)-action on \(P_{\gamma(E)}(E)\) given in the proof of \cite[Proposition 2.12]{stolz:_concor} that \(\phi\) is \(G(n,\gamma(E))\)-equivariant.
 Hence, the lemma is proved.
\end{proof}

\begin{constr}
\label{sec:semi-free-circle-7}
  If \(M\) is a connected free \(S^1\)-manifold, then there is an isomorphism \(TM\rightarrow p^*(T(M/S^1)\oplus \R)\).
Hence, by Lemma~\ref{sec:semi-free-circle-5}, one gets a equivariant \(\gamma(M/S^1)\)-structure on \(M\) from the canonical \(\gamma(M/S^1)\)-structure on \(M/S^1\).

We want to extend this construction to connected semi-free \(S^1\)-manifolds \(M\) with \(\codim M^{S^1}\geq 4\).
In this case we let \(\gamma(M/S^1)=\gamma((M-M^{S^1})/S^1)\).
By the assumption on the codimension of the fixed point set, the inclusion \(M-M^{S^1}\rightarrow M\) is three-connected.

In particular, \(\pi_i(M-M^{S^1})\cong \pi_i(M)\), for \(i=1,2\).
Hence, it follows from comparing the exact homotopy sequences for the
fibrations \(P_{O(n)}(T(M-M^{S^1}))/\langle r \rangle\rightarrow
M-M^{S^1}\) and
\(P_{O(n)}(TM)/\langle r \rangle\rightarrow M\) that the group
extensions
\(\pi_1(P_{O(n)}(T(M-M^{S^1}))/\langle r \rangle)\rightarrow
\pi_1(M-M^{S^1})\) and \(\pi_1(P_{O(n)}(TM)/\langle r
\rangle)\rightarrow \pi_1(M)\) are isomorphic.
Moreover, since \(M-M^{S^1}\subset M\) is open it follows that the
orientation character of \(M-M^{S^1}\) factors through the orientation
character of \(M\).
Therefore we have \(\gamma(M)=\gamma(M-M^{S^1})\).

By the case of a free action, we have a \(S^1\)-equivariant
\(\gamma(M/S^1)\)-structure on \(M-M^{S^1}\).
By part (\ref{item:2}) of  Lemma~\ref{sec:semi-free-circle-5}, this \(\gamma(M/S^1)\)-structure extends to an \(\gamma(M/S^1)\)-structure \(P_{\gamma(M/S^1)}\) on all of \(M\).
We will show that the \(S^1\)-action on \(M\) lifts into this \(\gamma(M/S^1)\)-structure on \(M\).
Since \(\gamma(M/S^1)\) is discrete, \(P_{\gamma(M/S^1)}\) is a covering of \(P_{O(n)}(TM)\) if \(w\neq 0\) or \(P_{SO(n)}(TM)\) if \(w=0\).
Therefore the \(S^1\)-action on \(P_{O(n)}(TM)\) or on \(P_{SO(n)}(TM)\) induces an \(\R\)-action on  \(P_{\gamma(M/S^1)}\).
By construction, the restriction of this \(\R\)-action to \(P_{\gamma(M/S^1)}|_{M-M^{S^1}}\) factors through \(S^1\).
Since \(M-M^{S^1}\) is dense in \(M\), this also holds for the \(\R\)-action on \(P_{\gamma(M/S^1)}\).
Hence, the \(S^1\)-action lifts into \(P_{\gamma(M/S^1)}\).

So on every connected semi-free \(S^1\)-manifold \(M\) with \(\codim M^{S^1}\geq 4\) there is a preferred equivariant \(\gamma(M/S^1)\)-structure.
\end{constr}

We need one more definition from \cite{stolz:_concor}.

\begin{definition}
  For \(n\geq 0\) and a discrete supergroup \(\gamma\) let \(R_n(\gamma)\) be the bordism group of \(n\)-dimensional \(\gamma\)-manifolds with positive scalar curvature metrics on their boundary, i.e. the objects of this bordism groups are pairs \((M,h)\) where \(M\) is an \(n\)-dimensional \(\gamma\)-manifold possibly with boundary and \(h\) a metric of positive scalar curvature on \(\partial M\). Two pairs \((M,h)\) and \((M',h')\) are identified if
  \begin{enumerate}
  \item there is a \(n\)-dimensional oriented manifold \(V\) with \(\partial V=-\partial M\amalg \partial M'\) and a positive scalar curvature metric on \(V\) which restricts to \(h\) and \(h'\) on the boundary respectively, and
  \item there is a \((n+1)\)-dimensional \(\gamma\)-manifold \(W\) with \(\partial W =M\cup_{\partial M} V\cup_{\partial M'}-M'\).
  \end{enumerate}
\end{definition}

In the above definition and the following we assume that all metrics on a manifold with boundary are product metrics near the boundary.

We also define equivariant bordism groups.

\begin{definition}
\label{sec:semi-free-circle}
  For \(n\geq 0\) and a discrete supergroup \(\gamma\) let \(\Omega_ {n,\geq 4}^{SF}(\gamma)\) be the bordism groups of closed \(n\)-dimensional  semi-free \(S^1\)-manifolds equipped with equivariant \(\gamma\)-structures and without fixed point components of codimension less than four.
Here we identify two manifolds \(M_1\) and \(M_2\) if there is a semi-free \(S^1\)-manifold with boundary, equivariant \(\gamma\)-structure and without fixed point components of codimension less than four such that \(\partial W=M_1\amalg -M_2\).
\end{definition}

Now we want to define a homomorphism \(\phi:\Omega^{SF}_{n,\geq 4}(\gamma)\rightarrow R_{n-1}(\gamma)\).

Let \(M\) be a \(S^1\)-manifold as in Definition \ref{sec:semi-free-circle}.
Then we can construct from \(M\) an \((n-1)\)-dimensional manifold \(N\) with boundary by removing from \(M\) an open tubular neighborhood of the fixed point set \(M^{S^1}\) and taking the quotient of the free \(S^1\)-action on this complement.
The boundary of this quotient is a disjoint union of \(\C P^k\)-bundles, \(k\geq 1\),  over the components of \(M^{S^1}\).
There is a natural \(\gamma\)-structure on this quotient because \((TM|_{M-M^{S^1}})/S^1=T((M-M^{S^1})/S^1)\oplus \R\).

We may choose a metric on \(M^{S^1}\) and a connection for the \(\C P^k\)-bundles over \(M^{S^1}\).
With these choices one can construct a connection metric \(h\) on \(\partial N\) with positive scalar curvature, such that the fibers of the bundle are up to scaling with a constant isometric to \(\C P^k\) with standard metric.
Note that the isotopy class of \(h\) does not depend on the above choices.
Therefore \((N,h)\) defines a well defined element of \(R_{n-1}(\gamma)\).

If \(W\) is an equivariant bordism without fixed point components of codimension less than four between \(M\) and another \(S^1\)-manifold \(M'\), then the quotient \(\tilde{W}\) of the free \(S^1\)-action on the complement of an open tubular neighborhood of the fixed point set in \(W\) gives a bordism between \(N\) and \(N'\).
The part of the boundary of \(\tilde{W}\) which does not belong to \(N\) or \(N'\) can be equipped with a metric of positive scalar curvature by the same argument as above.
Therefore the bordism class of \((N,h)\) depends only on the bordism class of \(M\).
Hence, we get the desired homomorphism \(\phi:\Omega^{SF}_{n,\geq 4}(\gamma)\rightarrow R_{n-1}(\gamma)\).

It has been shown by Stolz \cite{stolz:_concor} that a connected manifold \(M\) of dimension \(n\geq 5\) with boundary and a given metric \(h\) of positive scalar curvature on \(\partial M\) admits a metric of positive scalar curvature, which extends \(h\) and is a product metric near the boundary, if and only if \((M,h)\) equipped with the canonical \(\gamma(M)\)-structure represents zero in \(R_{n}(\gamma(M))\).

From this fact we get the following theorem.

\begin{theorem}
\label{sec:semi-free-circle-6}
  Let \(M\) be closed connected semi-free \(S^1\)-manifold \(M\) of dimension \(n\geq 6\) and \(\codim M^{S^1}\geq 4\).
  We equip \(M\) with the equivariant \(\gamma(M/S^1)\)-structure from Construction \ref{sec:semi-free-circle-7}.
  Then \(M\) admits an invariant metric of positive scalar curvature if and only if the class \([M]\in \Omega^{SF}_{n,\geq 4}(\gamma(M/S^1))\) can be represented by an \(S^1\)-manifold \(N\) which admits an invariant metric of positive scalar curvature.
\end{theorem}
\begin{proof}
  We will show that \(M\) admits an invariant metric of positive scalar curvature if and only if \(\phi([M])=0\in R_{n-1}(\gamma(M/S^1))\).
  From this the statement follows.

  At first assume that \(M\) admits an invariant metric of positive scalar curvature.
  By the proof of Gromov's and Lawson's surgery theorem \cite{MR577131} (see also \cite{MR962295}), this is the case if and only if it admits an invariant metric of positive scalar curvature which is a connection metric on a tubular neighborhood \(N\) of \(M^{S^1}\).
  Here the metrics on the fibers of \(N\rightarrow M^{S^1}\) are so called ``torpedo metrics''.
  
  Therefore after removing a small open tubular neighborhood \(N'\) of \(M^{S^1}\), the metric on  \(M-N'\) is a product metric near the boundary.
  Its restriction to the boundary is a connection metric with fibers isometric to a round sphere.
  Hence, the volume of the \(S^1\)-orbits is constant in a small neighborhood of the boundary.
  Therefore, by Theorem \ref{sec:pre-1} and the definition of \(\phi\), we have \(\phi([M])=0\in R_{n-1}(\gamma(M/S^1))\).

  Now assume that \(\phi([M])=0\).
  Then, by the result of Stolz mentioned above, there is a metric of positive scalar curvature on \((M-N')/S^1\) whose restriction to the boundary is a connection metric with fibers isometric to \(\C P^k\), \(k>1\).
  By Theorem~\ref{sec:pre-1}, there is an invariant metric on \(M-N'\)
  whose restriction to the boundary is a connection metric for the
  bundle \(\partial(M-N')\rightarrow \partial(M-N')/S^1\).
  Since the fibers of \(\partial(M-N')/S^1\rightarrow M^{S^1}\) are
  (up to scaling) isometric to \(\C P^k\) with standard metric, the
  fibers of \(\partial (M-N')\rightarrow M^{S^1}\) are isometric to
  spheres \((S^{2k+1},g)\), where the metric \(g\) can be constructed
  from the standard round metric by shrinking the orbits of the free
  linear \(S^1\)-action on \(S^{2k+1}\).
  Moreover, the metric on \(\partial (M-N')\) is a connection metric
  for the bundle \(\partial (M-N')\rightarrow M^{S^1}\).
  Therefore the  metric on \(\partial (M-N')\) is isotopic to a connection metric with fibers isometric to round spheres.
  Hence, we can glue in the normal disc bundle of \(M^{S^1}\subset M\) equipped with an appropriate metric to get an invariant metric of positive scalar curvature on \(M\).
\end{proof}

Now we want to discuss the special case, where \(M\) is a simply
connected semi-free \(S^1\)-manifold with \(\codim M^{S^1}\geq 4\).
In this case \(M/S^1\) and \((M-M^{S^1})/S^1\) are also simply
connected (see \cite[Corollary 6.3, p. 91]{MR0413144}).
Hence, a \(\gamma(M/S^1)\)-structure on \(M/S^1\) is a Spin-structure on \((M-M^{S^1})/S^1\) if \((M-M^{S^1})/S^1\) admits a Spin-structure or an orientation on \((M-M^{S^1})/S^1\), otherwise.

Note that \((M-M^{S^1})/S^1\) admits a Spin-structure if and only if \(M\) admits a Spin-structure and the \(S^1\)-action on \(M\) is of even type.

Denote by \(\Omega_{n}^{SO,SF}\) the bordism group of closed oriented semi-free \(S^1\)-manifolds and by \(\Omega_n^{\text{Spin},\text{even},SF}\) the bordism group of closed semi-free \(S^1\)-manifolds with Spin-structure and an \(S^1\)-action of even type.
Then we have:

\begin{theorem}
\label{sec:semi-free-circle-2}
  Let \(M\) be a closed simply connected semi-free oriented \(S^1\)-manifold \(M\) of dimension \(n\geq 6\) and \(\codim M^{S^1}\geq 4\).
 If \(M\) is not spin or spin with the \(S^1\)-action of odd type, then \(M\) admits an invariant metric of positive scalar curvature if and only if the class \([M]\in \Omega^{SO,SF}_n\) can be represented by an \(S^1\)-manifold \(N\) with \(\codim N^{S^1}\geq 4\) and an invariant metric of positive scalar curvature. 
  If \(M\) is spin and the \(S^1\)-action is of even type then the same holds with the oriented equivariant bordism ring replaced by the spin equivariant bordism ring.
\end{theorem}
\begin{proof}
  At first assume that \(M\) is spin and the \(S^1\)-action on \(M\) is of even type.
  Since on a Spin-manifold with even semi-free \(S^1\)-action there are no fixed point components of codimension two, we have \(\Omega_{n,\geq 4}^{SF}(\gamma(M/S^1))=\Omega_n^{\text{Spin},\text{even},SF}\).
  Therefore the theorem follows from Theorem~\ref{sec:semi-free-circle-6} in this case.

  Next assume that we are in the other case.
  Then \(\Omega_n^{SF}(\gamma(M/S^1))\) is just the bordism group \(\Omega_{n,\geq 4}^{SO,SF}\) of closed oriented semi-free \(S^1\)-manifolds without fixed point components of codimension less than four.
By the proof of Theorem~\ref{sec:semi-free-circle-6}, \(M\) admits an
invariant metric of positive scalar curvature if and only if \(\phi([M])=0\).
  Therefore it is sufficient to show that \(\phi([M])\) depends only on the image of \([M]\) under the natural map \(\Omega_{n,\geq 4}^{SO,SF}\rightarrow \Omega_{n}^{SO,SF}\).
  This can be shown as follows.

  Let \(W\) be a bordism between the semi-free \(S^1\)-manifolds \(M_1\) and \(M_2\) such that all fixed point components of \(W\) of codimension two do not meet the boundary.
  Then we can cut out these components to get a bordism without fixed point components of codimension two between \(M_1\) and \(M_2\amalg N_1\amalg \dots \amalg N_k\), where the \(N_i\) are free \(S^1\)-manifolds.
The claim follows if we show that \(\phi([N_i])\) vanishes for all \(i\).
The orbit spaces of the \(N_i\) are closed
manifolds. Moreover, every class in \(\Omega_{n-1}^{SO}\) can be represented
by a manifold which admits a metric of positive scalar curvature
(see \cite[Proof of Theorem C]{MR577131}).
Therefore it follows from the definition of the groups
\(R_{n-1}(\gamma)\) that \(\phi([N_i])=[N_i/S^1]=0\).
\end{proof}

\subsection{The non-spin case}
\label{sec:non-spin-case}

Our next goal is to prove the following theorem.

\begin{theorem}
 \label{sec:semi-free-circle-1}
  Let \(M\) be a closed simply connected semi-free \(S^1\)-manifold of dimension \(n>5\).
  If \(M\) is not spin or spin and the \(S^1\)-action is odd, then
   the equivariant connected sum of two copies of \(M\) admits an invariant metric of positive scalar curvature.
\end{theorem}

The proof of this theorem is divided into two cases:
\begin{enumerate}
\item \(M\) has a fixed point component of codimension two.
\item All fixed point components of \(M\) have codimension at least four.
\end{enumerate}

In the first case, the theorem follows from Theorem \ref{sec:constr-invar-psc-1}.
In the second case, the idea for the proof of Theorem~\ref{sec:semi-free-circle-1} is to show that the class of \(2M\) in \(\Omega_*^{SO,SF}\) can be represented by a manifold which admits an invariant metric of positive scalar curvature and does not have fixed point components of codimension less than four.
Then it follows from Theorem \ref{sec:semi-free-circle-2}.

We prepare the proof by describing some structure results about the ring \(\Omega_*^{SO,SF}\).
For background information on these results see for example \cite{MR0176478}, \cite[Chapter XV]{MR1413302} or \cite{MR0278338}.
At first there is an exact sequence 
\begin{equation*}
  \xymatrix{
    0\ar[r]&\Omega_*^{SO,SF}\ar^(.6){\lambda}[r]& F_*\ar[r]^(.3){\mu}&\Omega^{SO}_{*-2}(BU(1))\ar[r]&0.}
\end{equation*}

Here, \(F_*=\bigoplus_{n\geq 0}\Omega^{SO}_{*-2n}(BU(n))\) is the bordism ring of complex vector bundles over some base spaces.
The map \(\lambda\) sends a semi-free \(S^1\)-manifold \(M\) to the normal bundle of its fixed point set.
Note that this bundle is naturally isomorphic to a complex \(S^1\)-vector bundle over \(M^{S^1}\) of the form \(V\otimes \rho\), where \(V\) is a complex vector bundle over \(M^{S^1}\) with trivial \(S^1\)-action and \(\rho\) is the standard one-dimensional complex \(S^1\)-representation.
Therefore \(F_*\) might we viewed as the bordism group of fixed point data of semi-free oriented \(S^1\)-manifolds.

In this picture the summand \(F_0\) is isomorphic to the subgroup of
\(\Omega_*^{SO,SF}\) consisting of bordism classes of manifolds with trivial \(S^1\)-action. 
An isomorphism is induced by \(\lambda\).

The map \(\mu\) can be described as follows.
The restriction of the multiplication with elements of \(S^1\subset \C\) on a complex vector bundle to its sphere bundle defines a free \(S^1\)-action.
So we get a map from \(F_*\) to the bordism group of free \(S^1\)-manifolds.
Since every free \(S^1\)-manifold is a principal \(S^1\)-bundle over its orbit space, we may identify the bordism group of free \(S^1\)-actions with \(\Omega^{SO}_*(BU(1))\).
The map \(\mu\) is the composition of the above map and this identification.

\(\Omega_*^{SO,SF}\) and \(F_*\) are actually algebras over \(\Omega^{SO}_*\), with multiplication given by direct products.
Hence, \(\lambda\) is an algebra homomorphism.
Moreover, \(F_*\) is isomorphic to \(\Omega_*^{SO}[X_i;i\geq 0]\),
where \(X_i\) can be identified with the dual Hopf bundle over the
\(i\)-dimensional complex projective space.
Note that, by our grading of \(F_*\), \(X_i\) has degree \(2i+2\).

We denote by \(\C P^n(\rho)\) the \(n\)-dimensional complex projective space equipped with the \(S^1\)-action induced by the representation \(\rho\oplus \C^n\). Here \(S^1\) acts trivially on the \(\C^n\) summands.
Then we have
\begin{equation*}
  \lambda(\C P^n(\rho))= X_{n-1}+(-1)^nX^{n}_0.
\end{equation*}
Hence, \(F_*\) is isomorphic to \(\Omega_*^{SO}[X_0,\lambda(\C P^n(\rho)),n\geq 2]\).

For the proof of Theorem~\ref{sec:semi-free-circle-1} we need the following two lemmas.
To motivated our first lemma, consider the situation where \(M\) is a
\(2n\)-dimensional
semi-free \(S^1\)-manifold with isolated fixed points.
Then the normal bundle of the fixed point set in \(M\) is trivial.
Moreover, \(\lambda(M)=X_0^n\lambda(S)\), where \(S\) is a zero
dimensional \(S^1\)-manifold with trivial action.
By \cite[Corollar 6.24]{MR1150492}, the signature of \(M\) vanishes.
Therefore, by \cite[Corollary 6.23]{MR1150492}, \(S\) represents zero
in \(\Omega_0^{SO}\).

The following lemma is a generalization of this fact to
semi-free \(S^1\)-manifolds \(M\) such that the normal bundle of the fixed point
set \(M^{S^1}\) has a section which is nowhere zero.

\begin{lemma}
\label{sec:semi-free-circle-4}
  Let \(L\in F_*\). If \(X_0 L\in \kernel \mu = \Omega_*^{SO,SF}\), then the same holds for \(L\).
  In other words, \(L\) is the fixed point data of some semi-free \(S^1\)-manifold \(\tilde{L}\).
  Moreover, \(\tilde{L}\) is mapped to zero by the forgetful map \(\Omega_*^{SO,SF}\rightarrow \Omega_*^{SO}\).
\end{lemma}
\begin{proof}
  By Theorem 17.5 of \cite[p. 49]{MR0176478}, a class \([\phi:L'\rightarrow BU(1)]\in \Omega_*^{SO}(BU(1))\) represents zero if and only if its Stiefel-Whitney and Pontrjagin numbers vanish, i.e. the characteristic numbers of the form
  \begin{align*}
    \langle \phi^*(x)^lp_I(L'),[L']\rangle&&\text{and}&&\langle \phi^*(x)^l w_I(L'),[L']\rangle
  \end{align*}
vanish,
where \(x\) is a generator of \(H^2(BU(1))\) and \(p_I(L')\) and \(w_I(L')\) are products of the Pontrjagin classes and Stiefel--Whitney classes of \(L'\), respectively.
  Now note that \(\mu(L)\) is represented by a sum of tautological bundles over the projectivizations \(P(E_i)\) of  complex vector bundles \(E_i\).
  The class \(\mu(X_0L)\) is represented by the sum of tautological bundles over the projectivizations \(P(E_i\oplus \C)\)  of  the sum \(E_i\oplus\C\)  of these complex vector bundles with a trivial line bundle.
  
  Let \(E\) be a complex vector bundle of dimension \(n\) over a
  connected manifold \(B\).
  Then for \(K=\mathbb{Z}_2\) or \(K=\Q\), we have
  \begin{equation*}
    H^*(P(E);K)=H^*(B;K)[x]/(\sum_{i=0}^nc_i(E)x^{n-i})
  \end{equation*}
and
  \begin{equation*}
    H^*(P(E\oplus \C);K)=H^*(B;K)[x]/((\sum_{i=0}^nc_i(E)x^{n-i})x),
  \end{equation*}
where \(x\) has degree two and is minus the first Chern-class of the tautological bundle over \(P(E)\) and \(P(E\oplus \C)\).

Let \(y\) be a generator of the top cohomology group of \(B\).
Then we can orient \(P(E)\) and \(P(E\oplus \C)\) in such a way that
\begin{equation*}
  \langle yx^{n-1},[P(E)]\rangle=1=\langle yx^{n},[P(E\oplus \C)]\rangle.
\end{equation*}

Hence, if \(f(x)\) is a power series with coefficients in \(H^*(B;K)\), then we have
\begin{equation*}
  \langle f(x),[P(E)]\rangle=\langle xf(x),[P(E\oplus \C)]\rangle.
\end{equation*}

The total Pontrjagin and Stiefel-Whitney classes of \(P(E)\) are given by
\begin{align*}
  p(P(E))&=p(B)\prod_i(1+(a_i+x)^2),& w(P(E))&=w(B)\prod_i(1+a_i+x),
\end{align*}
where the \(a_i\) are the formal roots of the Chern classes of \(E\).

Therefore there are the following relations between the total Pontrjagin-classes and Stiefel-Whitney-classes  of \(P(E)\) and \(P(E\oplus \C)\) (both viewed as power series in \(x\) with coefficients in \(H^*(B;K)\)) 
\begin{align*}
  p(P(E\oplus \C))&=p(P(E))(1+x^2),& w(P(E\oplus \C))&=w(E)(1+x).
\end{align*}
This implies
\begin{align*}
  p(P(E))&=p(P(E\oplus \C))(\sum_{i=0}^{\infty}(-x^2)^i),& w(P(E))&=w(E\oplus \C)(\sum_{i=0}^{\infty}(-x)^i).
\end{align*}

For a power series \(f(x)\) and a finite sequence \(I\) of positive integers let \(f_I=\prod_{i\in I}f_i\), where \(f_i\) denotes the degree \(i\) part of \(f\).
With this notation we have
\begin{align*}
  \langle x^lp_I(P(E)),[P(E)]\rangle&=\langle x^{l+1} p_I(P(E)),[P(E\oplus\C)]\rangle\\
  &=\langle x^{l+1}(p(P(E\oplus \C))(\sum_{i=0}^{\infty}(-x^2)^i))_I,[P(E\oplus \C)]\rangle\\
  &= \sum_J a_J\langle x^{l+1+d_J}p_J(P(E\oplus \C)),[P(E\oplus \C)]\rangle,
\end{align*}
where \(a_J,d_J\in \mathbb{Z}\) only depend on \(I\) but not on \(E\).
A similar calculation shows
\begin{align*}
  \langle x^lw_I(P(E)),[P(E)]\rangle&=\sum_J b_J\langle x^{l+1+d_J'}w_J(P(E\oplus \C)),[P(E\oplus \C)]\rangle.
\end{align*}

Since \(\mu(L)\) and \(\mu(X_0L)\) are linear combinations of some projectivizations of complex vector bundles, it follows that \(\mu(L)=0\) if \(\mu(X_0L)=0\).
This proves that there is a \(\tilde{L}\) with \(L=\lambda(\tilde{L})\).

Now let \(E\) be the principal \(S^1\)-bundle associated to the tautological bundle over \(\C P^1(\rho)\).
Then the \(S^1\)-action on \(\C P^1(\rho)\) lifts into \(E\) in such a way that the action on the fiber over one of the fixed points in \(\C P^1(\rho)\)  is trivial and multiplication on the fiber over the other fixed point.
This action induces an \(S^1\)-action on \(E\times_{S^1}\tilde{L}\), for this action we have
\begin{equation*}
  \lambda(E\times_{S^1}\tilde{L})= X_0 L - X_0 \tilde{L}',
\end{equation*}
where \(\tilde{L}'\) is \(\tilde{L}\) with trivial \(S^1\)-action.
Therefore, we have
\begin{equation*}
  \mu(X_0)\tilde{L}'= \mu(X_0 L) - \mu(\lambda(E\times_{S^1}\tilde{L}))=0.
\end{equation*}
But \(\mu(X_0)\) is part of a \(\Omega_*^{SO}\)-basis of \(\Omega_*^{SO}(BU(1))\).
Hence \(\tilde{L}'\) represents zero in \(\Omega_*^{SO}\) and the lemma is proved.
\end{proof}

On \(\Omega_*^{SO,SF}\) there is an involution \(\iota\) which sends a semi-free \(S^1\)-manifold \(M\) to itself equipped with the inverse \(S^1\)-action.
Similarly there is an involution on \(F_*\), which sends a complex vector bundle to its dual and changes the orientation of the base space if the fiber dimension is odd.
Since \(\lambda\) is compatible with these two involutions, we denote the involution on \(F_*\) also by \(\iota\).

\begin{lemma}
\label{sec:semi-free-circle-3}
  Let \([M]\in \Omega_*^{SO,SF}\). Then we have
  \begin{equation*}
    \iota([M])=
    \begin{cases}
      [M]&\text{if } \dim M \equiv 0 \mod 4\\
      -[M]&\text{if } \dim M \equiv 2 \mod 4.
    \end{cases}
  \end{equation*}
\end{lemma}
\begin{proof}
  At first note that there is an equivariant diffeomorphism between \(\C P^n(\rho)\) and \(\iota(\C P^n(\rho))\) given by complex conjugation.
  It is orientation preserving if and only if \(n\) is even.
  Moreover, \(\iota(X_0)=-X_0\).

  As noted before we can write \(\lambda([M])\) as a linear combination of products 
  \begin{equation*}
X_0^k\lambda(\prod_i \C P^{n_i}(\rho))\times \beta=P\times \beta    
  \end{equation*}
  with \(\beta\in \Omega_*^{SO}\).
  By the above remark we have \(\iota(P\times\beta)=(-1)^{k+l}P\times \beta\), where \(l\) is the number of odd \(n_i\) appearing in the product.
  If \(\dim \beta \not\equiv 0\mod 4\), then \(\beta\) is of order two.
  Therefore we have \(\iota(P\times\beta)=-P\times \beta=P\times \beta\) in this case.

  If \(\dim \beta \equiv 0\mod 4\), we have \(\dim M - 2(k+l) \equiv \dim \beta \equiv 0 \mod 4\).
  Therefore the statement follows.
\end{proof}

The following construction provides examples of semi-free \(S^1\)-manifolds with invariant metrics of positive scalar curvature and without fixed point components of dimension less than four.

\begin{constr}
\label{sec:non-spin-case-1}
 Let \(\gamma\) be a complex line bundle over an oriented manifold \(M_0\).
Let also \(M\) be an oriented \(S^1\times S^1\)-manifold such that the \(S^1\)-actions induced by the inclusions of the \(S^1\)-factors in \(S^1\times S^1\) are semi-free.
Denote \(M\) equipped with the first (resp. second) of these actions by \(M_1\) (resp. \(M_2\)).
  Then the multiplication on \(\gamma\) induces an \(S^1\)-action on
  the projectivization \(P(\gamma\oplus\C)\).
  
  This action can be lifted in two different ways in the tautological
  bundle
  \begin{equation*}
    \gamma''=\{(x,v)\in P(\gamma\oplus \C)\times(\gamma\oplus \C);\;
    v\in x\}.
  \end{equation*}

The first action is given by \(g(x,v)=(g \cdot x,g\cdot v)\), for
\(g\in S^1\) and \((x,v)\in \gamma''\).
Here \(\cdot\) denotes the \(S^1\)-action induced by multiplication on
\(\gamma\).

The second action is given by \(g(x,v)=(g\cdot x,
g\cdot(\lambda(g)v))\).
Here \(\cdot\) is a defined as above and \(\lambda(g)v\) is given by
complex multiplication of \(v\in \gamma\oplus \C\) with
\(\lambda(g)=g^{-1}\).

By dualizing these two actions we get two \(S^1\)-actions on the dual
\(\gamma'\) of \(\gamma''\).

  Let \(E\) be the principal \(S^1\times S^1\)-bundle associated to \(\gamma'\oplus\gamma'\).
  Then from the two \(S^1\)-actions on \(\gamma'\), we see that the
  \(S^1\)-action on \(P(\gamma\oplus \C)\) lifts into \(E\) in such a way that the weights of the restriction of the lifted \(S^1\)-action to a fiber over the two fixed point components in \(P(\gamma\oplus \C)\) are given by \((1,0)\) and \((0,-1)\), respectively.

The \(S^1\)-action on \(E\) induces a semi-free \(S^1\)-action on \(\Gamma(\gamma,M)=E\times_{S^1\times S^1}M\) and we have
\begin{equation*}
  \lambda(\Gamma(\gamma,M))= [\gamma]\lambda(M_1)-[\gamma]\lambda(\iota(M_2)),
\end{equation*}
where \([\gamma]\in\Omega_*^{SO}(BU(1))\) is represented by the bundle \(\gamma\).
If the \(S^1\)-action on all components of \(M_i\), \(i=1,2\), is non-trivial, then \(\Gamma(\gamma,M)\) does not have any fixed point components of codimension less than four.

Moreover, \(\Gamma(\gamma,M)\) admits an invariant metric of positive scalar curvature.
This can be seen as follows.
Since there is an invariant metric of positive scalar curvature on the fibers \(\C P^1(\rho)\) of \(P(\gamma\oplus \C)\), it follows that there is an \(S^1\)-invariant connection metric of positive scalar curvature on \(P(\gamma\oplus \C)\).
Therefore it follows from Theorem~\ref{sec:pre-1} that \(E\) admits an \(S^1\times S^1\times S^1\)-invariant metric of positive scalar curvature.
  Hence, \(E\times M\) admits an \(S^1\times S^1\times S^1\)-invariant metric of positive scalar curvature.
  Now it follows, again from Theorem~\ref{sec:pre-1}, that \(\Gamma(\gamma,M)=E\times_{S^1\times S^1}M\) admits an \(S^1\)-invariant metric of positive scalar curvature.
\end{constr}

Let \(\gamma\) and \(M_0\) as in the above construction and \(M\) a semi-free \(S^1\)-manifold.
We equip \(M\) with a \(S^1\times S^1\)-action induced by the homomorphism
 \(S^1\times S^1\rightarrow S^1\), \((z_1,z_2)\mapsto z_1z_2\)
and define \(\Delta(\gamma,M)=\Gamma(\gamma,M)\).
Then we have
\begin{equation*}
  \lambda(\Delta(\gamma,M))= [\gamma]\lambda(M)-[\gamma]\lambda(\iota(M)).
\end{equation*}

Now we can prove Theorem~\ref{sec:semi-free-circle-1}.

\begin{proof}[Proof of Theorem~\ref{sec:semi-free-circle-1}]
  Let \(M\) be a semi-free \(S^1\)-manifold.
  By Theorem~\ref{sec:constr-invar-psc-1}, we may assume that \(M\) does not have fixed point components of codimension less than four.
  Denote by \(n\) the dimension of \(M\) and assume \(n\geq 6\).

  If \(n\equiv 1,3\mod 4\), then the class \([M]\in \Omega_*^{SO,SF}\) is of order two, because all torsion elements of \(\Omega_*^{SO}\) have order two, \(\Omega_*^{SO}/\text{torsion}\) is concentrated in degrees divisible by four and the generators of \(F_*\) all have even degrees.
  Therefore, by Theorem~\ref{sec:semi-free-circle-2}, the theorem follows in this case.

  Next assume that \(n\equiv 0\mod 4\).
  Then there are an \(L\) in the augmentation ideal of \(F_*=\Omega_*^{SO}[X_0,\lambda(\C P^n(\rho));n\geq 2]\) and \(\beta_J\in \Omega_*^{SO}\) such that
  \begin{equation*}
    \lambda(M)=\sum_J \beta_J \lambda(\prod_{i\in J} \C P^{i}(\rho)) + X_0 L.
  \end{equation*}
  Here, the sum is taken over all finite sequences \(J\) with at least two elements, i.e. the products \(\prod_{i\in J} \C P^{i}(\rho)\) consist out of at least two factors. 
  Since all these products admit invariant metrics of positive scalar curvature, by Theorem~\ref{sec:semi-free-circle-2}, we only have to deal with the case \(\lambda(M)= X_0 L\).
 It follows from Lemma~\ref{sec:semi-free-circle-4} that there is a \(S^1\)-manifold \(\tilde{L}\) of dimension \(n-2\) with \(\lambda(\tilde{L})=L\).
 Since \(L\) is contained in the augmentation ideal of \(F_*\), we may assume that \(\tilde{L}\) does not have any components with trivial \(S^1\)-action.
 Hence, by Lemma~\ref{sec:semi-free-circle-3}, we have \(2\lambda(M)=\lambda(\Delta(X_0,\tilde{L}))\).
 Therefore the statement follows from Theorem~\ref{sec:semi-free-circle-2}.

 Now assume that \(n\equiv 2\mod 4\).
 Then as in the previous case, me may assume that
 \begin{equation*}
   \lambda(M)=X_0\lambda(\tilde{L}),
 \end{equation*}
 where \(\tilde{L}\) is a semi-free \(S^1\)-manifold without components on which \(S^1\) acts trivially.
 By a similar argument as above, we have
 \begin{equation*}
   \lambda(\tilde{L})=\sum_J \beta_J\lambda(\prod_{i\in J} \C P^{i}(\rho)) + X_0\lambda(L')
 \end{equation*}
  with \(L'\in \Omega_{n-4}^{SO,SF}\), \(\beta_J\in \Omega_*^{SO}\). 
Here the sum is taken over all finite sequences \(J\) with at least one element.
At first we show  that we may assume \(L'=0\). 
Let \(L''\) be the union of those components of \(L'\) on which the \(S^1\)-action is non-trivial.
Then we have, for 
\begin{equation*}
  N=2\C P^2(\rho)\times L''-\Delta(X_1,L''),
\end{equation*}
  \(\lambda(N)=2X_0^2\lambda(L'')=2X_0^2\lambda(L') \) because \(L'-L''\) has order two in \(\Omega_*^{SO}\).
By a similar construction, one sees that we may assume that the products \(\prod_{i\in J}\C P^{i}(\rho)\) do not have factors of odd complex dimension.

The next step is to show that we may assume that all these products consist out of exactly one factor.

To see this note that one can equip \(\prod_{i\in J}\C P^{i}\) with two semi-free commuting \(S^1\)-actions, namely \(\prod_{i\in J}\C P^{i}(\rho)\) and \(\prod_{i\in J-\{i_0\}}\C P^{i}\times \C P^{i_0}(\rho)\).
Therefore, by Lemma~\ref{sec:semi-free-circle-3}, we have
\begin{align*}
  \lambda(\Gamma(X_0,\prod_{i\in J} \C P^{i}))&=X_0\lambda( \prod_{i\in J} \C P^{i}(\rho)) - X_0\lambda( \C P^{i_0}(\rho))\times \prod_{i\in J-\{i_0\}} \C P^{i}.
\end{align*}
Hence, after adding several multiples of \(\Gamma(X_0,\prod_{i\in J} \C P^{i})\), we may assume that
\begin{equation*}
  \lambda(2M)=X_0\sum_{k=1}^{(n-2)/4} 2\beta_k \lambda(\C P^{2k}(\rho)),
\end{equation*}
with \(\beta_k\in \Omega_*^{SO}\).
Since all torsion elements in \(\Omega_*^{SO}\) are of order two, this sum depends only on the equivalence classes of the \(\beta_k\) in \(\Omega_*^{SO}/\text{torsion}\).

\(\Omega_*^{SO}/\text{torsion}\) is a polynomial ring over \(\mathbb{Z}\) with one generator \(Y_{4k}\) in each dimension divisible by four.
We will construct a non-trivial semi-free \(S^1\)-action on each of the \(Y_{4k}\).
The \(Y_{4k}\) can be chosen in such a way that they admit stably complex structures, that is they belong to the image of the natural map \(\Omega_*^U\rightarrow \Omega_*^{SO}\).
It has been shown by Buchstaber and Ray \cite{MR1813798}, \cite{MR1639388} that \(\Omega_*^U\) is generated by certain projectivizations of sums of complex line bundles over bounded flag manifolds.
These generators are toric manifolds and admit a non-trivial semi-free \(S^1\)-action induced by multiplication on one of the line bundles.
Hence, we may assume that each \(Y_{4k}\) admits a non-trivial semi-free \(S^1\)-action.
We denote the resulting \(S^1\)-manifold by \(Y_{4k}'\).

Next we show that we may assume
\begin{equation*}
  \lambda(2M)=X_0\left(\sum_{k=1}^{(n-2)/4-1}\sum_{J_k} 2a_{k,J_k}\prod_{i\in J_k}Y_{4i} \lambda(Y_{4k}')+2b\lambda(\C P^{(n-2)/2}(\rho))\right).
\end{equation*}
Here the second sum is over all finite sequences \(J_k\) of elements of the set \(\{1,\dots,k\}\) and \(a_{k,J_k},b\in \mathbb{Z}\).

Since each \(\beta_k\), \(k<(n-2)/4\), is a linear combination of products of the \(Y_{4k}\), this can be achieved by adding to \(2M\) multiples of manifolds of the form \(\Gamma(X_0,Y_{4i}\times \C P^{2k})\) and \(\Gamma(X_0,Y_{4i_1}\times Y_{4i_2})\), where \(S^1\times S^1\) acts factorwise on the product, i.e. the \(i\)-th \(S^1\)-factor acts on the \(i\)-th factor of the product.

The manifold \(\C P^{(n-2)/2}\) is indecomposable in \(\Omega_*^{SO}/\text{torsion}\).
Therefore it follows from the above structure results on \(\Omega_*^{SO}/\text{torsion}\) and Lemma~\ref{sec:semi-free-circle-4} that \(2M=0\) in \(\Omega_*^{SO,SF}\).
This proves the theorem.
\end{proof}

\subsection{The spin case}
\label{sec:spin-case}

In \cite{MR1758446} Lott constructed a generalized \(\hat{A}\)-genus for orbit spaces of semi-free even \(S^1\)-actions on \(\text{Spin}\)-manifolds.
If \(M/S^1\) has dimension divisible by \(4\), then \(\hat{A}(M/S^1)\)
is an integer.
In all other dimensions it vanishes.
Moreover, if \(M/S^1\) is a manifold without boundary, i.e. the
\(S^1\)-action on \(M\) is free or trivial, Lott's generalized
\(\hat{A}\)-genus of \(M/S^1\) coincides with the usual
\(\hat{A}\)-genus of \(M/S^1\).

He showed that for a semi-free even \(S^1\)-manifold \(M\), \(\hat{A}(M/S^1)\) vanishes if \(M\) admits an invariant metric of positive scalar curvature.
In this subsection we prove the following partial  converse to his result.

\begin{theorem}
 \label{sec:spin-case-1}
  Let \(M\) be a closed simply connected Spin-manifold of dimension \(n>5\)  with an even  semi-free  \(S^1\)-action.
  Then  \(\hat{A}(M/S^1)=0\) if and only if there is a \(k\in \N\) such that
   the equivariant connected sum of \(2^k\) copies of \(M\) admits an invariant metric of positive scalar curvature.
\end{theorem}

The strategy for the proof of Theorem~\ref{sec:spin-case-1} is the same as in the proof of Theorem~\ref{sec:semi-free-circle-1}.
That means that we construct generators for the kernel of the \(\hat{A}\)-genus which admit invariant metrics of positive scalar curvature.
To do so we first review some results about the ring \(\Omega_*^{\text{Spin},\text{even},SF}\otimes \mathbb{Z}[\frac{1}{2}]\) for proofs of these results with \(\mathbb{Z}[\frac{1}{2}]\) replaced by the rationals see \cite{MR882700}. The proofs there also work in our case since  \(\Omega_*^{\text{Spin}}\otimes \mathbb{Z}[\frac{1}{2}]\cong  \Omega_*^{SO}\otimes \mathbb{Z}[\frac{1}{2}]\).
There are exact sequences

\begin{equation*}
  \xymatrix{
    \Omega^{\text{Spin}}_{*-1}(BU(1))\otimes \mathbb{Z}[\frac{1}{2}]\ar[r]&\Omega_*^{\text{Spin},\text{even},SF}\otimes \mathbb{Z}[\frac{1}{2}]\ar^{\lambda}[d]& \\ & F^{\text{even}}_*\otimes \mathbb{Z}[\frac{1}{2}]\ar[r]^(.42){\mu}&\Omega^{SO}_{*-2}(BU(1))\otimes \mathbb{Z}[\frac{1}{2}]}
\end{equation*}

and

\begin{equation*}
  \xymatrix{
    0\ar[r]&\Omega_*^{\text{Spin},\text{odd},SF}\otimes \mathbb{Z}[\frac{1}{2}]\ar^(.57){\lambda}[r] & F^{\text{odd}}_*\otimes \mathbb{Z}[\frac{1}{2}]\ar[r]^(.4){\mu}&\Omega^{SO}_{*-2}(BU(1))\otimes \mathbb{Z}[\frac{1}{2}]\ar[r]&0}
\end{equation*}

Here the first map is the natural map from the bordism group of Spin-manifolds with free even action into the bordism group of semi-free actions and
\begin{align*}
  F^{\text{even}}_*&=\bigoplus_{k\geq 0}\Omega_{*-4k}^{SO}(BU(2k))\subset F_*,\\ F^{\text{odd}}_*&=\bigoplus_{k\geq 0}\Omega_{*-4k-2}^{SO}(BU(2k+1))\subset F_*.
\end{align*}
Moreover, the maps \(\lambda\) and \(\mu\) are the same as in the oriented case.

At first we describe another set of generators of \(F_*\otimes \mathbb{Z}[\frac{1}{2}]\) which are more suitable for the discussion of the spin case than the generators described in the previous section.
Denote by \(\tilde{X}_{2k+1}\), \(k>1\), the normal bundle of \(P(\C^2\oplus \bigoplus_{i=1}^{2k-1} \gamma\otimes \gamma)\) in \(P(\C^2\oplus \bigoplus_{i=1}^{2k} \gamma\otimes \gamma)\), where \(\gamma\) is the dual Hopf bundle over \(\C P^1\) and \(P(E)\) denotes the projectivization of the vector bundle \(E\).
Then we have
\begin{equation*}
  \langle c_1(\tilde{X}_{2k+1})^{2k+1},[P(\C^2\oplus \bigoplus_{i=1}^{2k-1} \gamma\otimes \gamma)]\rangle=4.
\end{equation*}
Therefore, by \cite[Theorem 18.1]{MR0176478}, the \(\tilde{X}_{2k+1}\), \(k>0\), together with the \(X_{2k}\), \(k\geq 0\) and \(X_1\) are a basis of  \(\Omega_*^{SO}(BU(1))\otimes \mathbb{Z}[\frac{1}{2}]\) as a \(\Omega_*^{SO}\otimes \mathbb{Z}[\frac{1}{2}]\)-module.
Hence, we have \(F_*\otimes\mathbb{Z}[\frac{1}{2}]=\Omega_*^{SO}[\frac{1}{2}, X_0, X_1, X_{2k}, \tilde{X}_{2k+1}; k\geq 1]\).

For \(k\geq 1\), let \(M_{2k+1}=\C P^{2k+1}(\rho)\) and \(M_{2k+2}=P(\C^2\oplus \bigoplus_{i=1}^{2k} \gamma\otimes \gamma)\), where \(S^1\) acts by multiplication on one of the \(\gamma\otimes \gamma\) summands.
Then we have 
\begin{equation*}
  \lambda(M_{2k+2})=\tilde{X}_{2k+1} - 4 X_1X_0^{2k}.
\end{equation*}
Moreover, let \(M_0\) be \(\mathbb{H} P^2=Sp(3)/(Sp(2)\times Sp(1))\) equipped with the semi-free \(S^1\)-action induced by the embedding
\begin{equation*}
  S^1\hookrightarrow \{I\}\times Sp(1)\hookrightarrow Sp(2)\times Sp(1)\hookrightarrow Sp(3).
\end{equation*}
This action has two fixed point components, namely \(\mathbb{H} P^1\) and an isolated fixed point.
The Chern classes of the normal bundle of \(\mathbb{H} P^1\) in \(\mathbb{H} P^2\) are given by
\begin{align*}
  c_1&=0&c_2&=-u,
\end{align*}
where \(u\) is a generator of \(H^4(\mathbb{H} P^1;\mathbb{Z})\).
Hence, it follows from a calculation of characteristic numbers (cf. Theorem 17.5 of \cite[p. 49]{MR0176478}) that
 \begin{equation*}
   \lambda(M_0)=-X_1^2+2\lambda(M_3)X_0 - \frac{1}{8}[K] X_0^2+ X_0^4,
 \end{equation*}
 where \(K\) is the Kummer surface.
Therefore we have
\begin{equation*}
 F_*\otimes\mathbb{Z}[\frac{1}{2}]=\Omega_*^{SO}[\frac{1}{2},X_0, X_1, \lambda(M_{k}); k\geq 0]/(R-X_1^2),  
\end{equation*}
 where \(R-X_1^2\) is the relation described above.

The following lemma shows that the manifolds \(M_i\) defined above are Spin-manifolds.
The \(S^1\)-action on \(M_i\) is even if and only if \(i=0\).

\begin{lemma}
\label{sec:spin-case-2}
  Let \(N\) be a Spin-manifold with an action of a torus \(T=S^1_1\times\dots\times S^1_k\) such that each factor \(S^1_i\) has a fixed point in \(N\) and \(f:E\rightarrow M\) be a principal \(T\)-bundle.
Then we have
\begin{equation*}
  w_2(E\times_T N)=f^*(w_2(M) + \sum_{i=1}^k \epsilon_i w_2(E_i)),
\end{equation*}
where \(E_i\) is the principal \(S^1_i\)-bundle
\begin{equation*}
  E/(S^1_1\times\dots\times S_{i-1}^1\times \{e\}\times S_{i+1}\times\dots\times S_k^1)\rightarrow M
\end{equation*}
and \(\epsilon_i\) is one or zero if the \(S^1_i\)-action on \(N\) is odd or even, respectively.
\end{lemma}
\begin{proof}
  The tangent bundle of \(E\times_T N\) is the sum of the pullback of the tangent bundle of \(M\) and the bundle \(TF\) along the fibers.
Let \(\phi:M\rightarrow BT\) be the classifying map of the principal bundle \(E\).
Then there is a pullback diagram.

\begin{equation*}
  \xymatrix{
    TF\ar[r]\ar[d]& TN_T\ar[d]\\
    E\times_T N\ar[r]\ar^{f}[d]& N_T\ar^{\pi}[d]\\
    M\ar^{\phi}[r]&BT}
\end{equation*}

Here \(TN_T\) and \(N_T\) denote the Borel constructions of \(TN\) and \(N\), respectively.
Therefore it is sufficient to show that \(w_2(TN_T)=\pi^*(\sum_{i=1}^k\epsilon_i x_i)\), where \(x_i\) is the generator of \(H^2(BS_i^1;\mathbb{Z}_2)\subset H^2(BT;\mathbb{Z}_2)\).

There is an exact sequence
\begin{equation*}
\xymatrix{
  0\ar[r]& H^2(BT;\mathbb{Z}_2)\ar^{\pi^*}[r]&H^2(N_T;\mathbb{Z}_2)\ar[r]& H^2(N;\mathbb{Z}_2)}
\end{equation*}
Hence, we have \(w_2(TN_T)\in \image \pi^*\).
Let \(\iota:BS_i^1\rightarrow BT\) be the inclusion. 
Then we have \(\iota^*N_T=N_{S^1_i}\).
Since there are \(S^1_i\)-fixed points in \(N\) there is a section to \(N_{S^1_i}\rightarrow BS_i^1\) induced by the inclusion of a fixed point \(pt\).
Therefore we have \(w_2(TN_{S^1_i})=\pi^*\sum a_jx_i\), where the \(a_j\) are the weights of the \(S^1_i\)-representation \(T_{pt}N\).
Now the \(S^1_i\)-action on \(N\) is even if and only if \(\codim N^{\mathbb{Z}_2}\equiv 0 \mod 4\). That is if and only if there is an even number of odd \(a_j\).
Therefore the statement follows.
\end{proof}

It follows from this lemma, that, for a complex line bundle \(\gamma\) over an oriented manifold \(N\) and a Spin-manifold \(M\) with semi-free \(S^1\)-action, \(\Delta(\gamma,M)\) is spin if and only if \(w_2(N)\equiv c_1(\gamma)\mod 2\).
Moreover, if \(M'\) is a Spin-manifold with two commuting semi-free \(S^1\)-actions, then \(\Gamma(\gamma,M')\) is spin if and only if  \(w_2(N)\equiv c_1(\gamma)\mod 2\) and both \(S^1\)-actions are even or both \(S^1\)-actions are odd.

\begin{proof}[Proof of Theorem~\ref{sec:spin-case-1}]
  We will show that for every \([M]\in \Omega_n^{\text{Spin},\text{even},SF}\), \(n\geq 6\), there is a \([M']\in \Omega_n^{\text{Spin},\text{even},SF}\) such that \(M'\) admits an invariant metric of positive scalar curvature and \(\lambda(2^k[M])=\lambda([M'])\) in \(F_*^{\text{even}}\otimes \mathbb{Z}[\frac{1}{2}]\).
Hence, \(2^{k+k'}[M]-2^{k'}[M']\) is bordant to a free  \(S^1\)-manifold \(N\).
After doing some surgeries we may assume that \(N/S^1\) is simply connected. 
 Because \(\hat{A}(M/S^1)\) is an invariant of the equivariant bordism type of \(M\) and agrees with the usual \(\hat{A}\)-genus if the \(S^1\)-action is free, it follows from Theorems~\ref{sec:semi-free-circle-2} and \ref{sec:pre-1} that \(2^{k+k'+1}M\) admits an invariant metric of positive scalar curvature if and only if \(2^{k+k'+1}\hat{A}(M/S^1)=2\hat{A}(N/S^1)=0\).
Hence, Theorem~\ref{sec:spin-case-1} follows.

Now we turn to the construction of \(M'\).
At first assume that \(n\) is odd. Then \(F^{\text{even}}_n\otimes \mathbb{Z}[\frac{1}{2}]\) vanishes as in the non-spin case.
Therefore we may assume that \([M']=0\).

Next assume that \(n\equiv 0\mod 4\).
Then there are \(\alpha_J,\beta_J\in \Omega^{\text{Spin}}_*\otimes \mathbb{Z}[\frac{1}{2}]\) and an \(L\in F_*^{\text{odd}}\otimes \mathbb{Z}[\frac{1}{2}]\) such that
\begin{equation*}
  \lambda(M)=\sum_J \alpha_J \lambda(\prod_{i\in J}M_i) + X_1 \sum_{J'} \beta_{J'}\lambda(\prod_{i\in J'}M_i) + X_0 L.
\end{equation*}
Here the sums are taken over all finite sequences \(J\) and \(J'\) with at least two elements or one element, respectively.
Since all \(M_i\) admit invariant metrics of positive scalar curvature we may assume that \(\alpha_J=0\) for all \(J\).

The \(S^1\)-action on \(\sum_{J'} \beta_{J'}\prod_{i\in J'}M_i\) is of odd type.
Therefore each product \(\prod_{i\in J'}M_i\) contains a factor \(M_{i_0}\) with \(i_0>0\).

Let \(E\) be the principal \(S^1\)-bundle associated to the dual of the tautological bundle \(\gamma\)  over \(\C P^2(\rho)\).
The \(S^1\)-action on \(\C P^2(\rho)\) lifts into \(E\) such that the action on the fiber over the isolated fixed point is trivial.
Since the \(S^1\)-action on \(M_{i_0}\) is odd, \(E\times_{S^1}M_{i_0}\) is a Spin-manifold by Lemma~\ref{sec:spin-case-2}.
Moreover, it follows from a calculation of characteristic numbers that
\begin{equation*}
  \lambda(E\times_{S^1}M_{i_0})= X_1 \lambda(M_{i_0}) + X_0 L'
\end{equation*}
with some \(L'\in F^{\text{odd}}_*\otimes\mathbb{Z}[\frac{1}{2}]\).

Indeed, there are three fixed point components in \(E\times_{S^1}M_{i_0}\), namely the fiber over the isolated fixed point in \(\C P^2(\rho)\), and two components which are bundles over the two-dimensional fixed point component of \(\C P^2(\rho)\) with fibers the fixed point components of \(M_{i_0}\).
The normal bundle of the first fixed point component is trivial.

The other fixed point components are diffeomorphic to cartesian products \(\C P^1\times F_i\), where \(F_i\), \(i=1,2\), are the fixed point components in \(M_{i_0}\).
Hence, their cohomology with coefficients in \(\Q\) or \(\mathbb{Z}_2\) is isomorphic to \(H^*(\C P^1)\otimes H^*(F_i)\). Moreover, there are \(k_i\in \mathbb{Z}\), \(i=1,2\), such that the first two Chern classes of their normal bundles are given by
\begin{align*}
  c_1&=k_ic_1(\gamma)+c_1(N(F_i,M_{i_0}))& c_2&=c_1(\gamma)c_1(N(F_i,M_{i_0})).
\end{align*}
The other Chern classes vanish because these fixed point components have codimension or dimension four, respectively.

Therefore, all characteristic numbers of \( \lambda(E\times_{S^1}M_{i_0})- X_1 \lambda(M_{i_0})\) involving Chern classes \(c_i\), \(i>1\), of the normal bundle of the fixed point components vanish (cf. Theorem 17.5 of \cite[p. 49]{MR0176478}).
Since these normal bundles have complex dimension greater than one, it follows that \( \lambda(E\times_{S^1}M_{i_0})- X_1 \lambda(M_{i_0})\) is contained in the ideal of \(F_*\otimes \mathbb{Z}[\frac{1}{2}]\) generated by \(X_0\).

By the same argument as in Construction~\ref{sec:non-spin-case-1}, \(E\times_{S^1} M_{i_0}\) admits an invariant metric of positive scalar curvature.
Therefore we may assume that all \(\beta_{J'}\) are zero.

Hence, by the same argument as in the proof of Lemma~\ref{sec:semi-free-circle-4}, there is a Spin-manifold \(\tilde{L}\)  with semi-free \(S^1\)-action of odd type such that \(\lambda(\tilde{L})=2^kL\).
Then \(\Delta(X_0,\tilde{L})\) is a Spin-manifold such that
\begin{equation*}
  \lambda(\Delta(X_0,\tilde{L}))=2^k(X_0L-X_0\iota(L))=2^{k+1} X_0 L.
\end{equation*}
Hence, we may assume that \(L=0\).
Therefore the claim follows in this case.

Next assume that \(n\equiv 2 \mod 4\).
Then there are \(\alpha_{k,l,J} \in \Omega^{\text{Spin}}_*\otimes \mathbb{Z}[\frac{1}{2}]\) such that
\begin{equation*}
  \lambda(M)=\sum_{k=0}^{n/2} \sum_{l=0}^1 \sum_J \alpha_{k,l,J} \lambda(\prod_{i\in J}M_i)X_0^kX_1^l.
\end{equation*}
We will show that we may assume that all \(\alpha_{k,l,J}\) vanish after adding Spin-manifolds with even semi-free \(S^1\)-actions which admit invariant metrics of positive scalar curvature.
In the case that \(k=l=0\) there is nothing to show.

Next assume that \(k=0\) and \(l=1\).
Then the dimension of \(\prod_{i\in J} M_i\) with \(\alpha_{k,l,J}\neq 0\) is congruent to \(2\mod 4\).
Moreover, \(\Delta(X_1\otimes X_1,\prod_{i\in J} M_i)\) is a Spin-manifold with semi-free \(S^1\)-action such that
\begin{equation*}
  \lambda(\Delta(X_1\otimes X_1,\prod_{i\in J} M_i))= 4 X_1 \lambda(\prod_{i\in J} M_i).
\end{equation*}
Therefore we may assume that \(\alpha_{0,1,J}\) vanishes for all \(J\).

Next assume that \(k>0\) is odd and \(l=0\).
Then the dimension of \(\prod_{i\in J} M_i\) with \(\alpha_{k,l,J}\neq 0\) is divisible by four and the action on this product is of odd type.
Therefore such a product contains at least one factor \(M_{i_0}\) with \(i_0>0\).

At first assume that \(i_0\) is odd.
Then for the Spin-manifold
\begin{equation*}
  N=2\C P^{k}(\rho)\times M_{i_0}-\Delta(X_{k-1},M_{i_0}).
\end{equation*}
we have \(\lambda(N)=-2X_0^k\lambda(M_{i_0})\).
Therefore we may assume that all \(\alpha_{k,l,J}\) with \(k>0\) odd and \(l=0\) vanish if \(J\) contains an odd number.

Now we turn to the case where \(i_0\) is even.
Then there is a second semi-free \(S^1\)-action on \(M_{i_0}\) induced from a lift of the \(S^1\)-action on \(\C P^1(\rho)\) to \(\gamma\otimes \gamma\).
The two \(S^1\)-actions commute and are both of odd type.
Let \(M_{i_0}'\) be \(M_{i_0}\) equipped with the second \(S^1\)-action.

Let
\begin{equation*}
  N= \C P^k(\rho)\times(M_{i_0}-M_{i_0}')- \Gamma(X_{k-1},M_{i_0}).
\end{equation*}
Then we have, by Lemma~\ref{sec:semi-free-circle-3},
\begin{equation*}
  \lambda(N)=-X_0^k\lambda(M_{i_0})+ X_0^{k+1}\lambda(L),
\end{equation*}
where \(L\) is \(\C P^{i_0-1}\) equipped with some semi-free \(S^1\)-action.
Hence, we may assume that all \(\alpha_{k,l,J}\) with \(k>0\) odd and \(l=0\) vanish.

Next assume that \(k>0\)  is even and \(l=1\).
 Then the dimension of \(\prod_{i\in J} M_i\) with \(\alpha_{k,l,J}\neq 0\) is congruent to \(2\mod 4\) and the action on this product is of odd type.

Hence,
\begin{equation*}
  N=2M_{k+2}\times \prod_{i\in J} M_i - \Delta(\tilde{X}_{2k+1},\prod_{i\in J} M_i)
\end{equation*}
is spin and
\begin{equation*}
  \lambda(N)=-8 X_1X_0^k\lambda(\prod_{i\in J}M_i).
\end{equation*}
Therefore we may assume that all \(\alpha_{k,l,J}\) with \(k>0\) even and \(l=1\) vanish.

Next assume that \(k>0\) is even and \(l=0\).
 Then the dimension of \(\prod_{i\in J} M_i\) with \(\alpha_{k,l,J}\neq 0\) is congruent to \(2\mod 4\) and the action on this product is of even type.
Therefore in this product there appears at least one factor \(M_{i_0}\) with \(i_0\) odd and a second factor \(M_{i_1}\) with \(i_1>0\).
Let \(E\) be the principal \(S^1\)-bundle associated to the tautological line bundle over \(\C P^k(\rho)\).
Then the action on \(\C P^k(\rho)\) lifts into \(E\) in such a way that it is trivial on the fibers of the fixed point component of codimension two in \(\C P^k(\rho)\) and multiplication on the fiber over the isolated fixed point.
Moreover, \(N= E\times_{S^1}M_{i_0}\) is a Spin-manifold with semi-free \(S^1\)-manifold such that \(\lambda(N)=X_0^k\lambda(M_{i_0})\) because \(\dim M_{i_0}\equiv 2 \mod 4\).
As in Construction~\ref{sec:non-spin-case-1}, one sees that \(N=E\times_{S^1}M_{i_0}\) admits an invariant metric of positive scalar curvature.
Hence, after adding multiples of \(\prod_{i\in J-\{i_0\}}M_i \times N\) to \(M\),
 we may assume that all \(\alpha_{k,l,J}\) with \(k>0\) even and \(l=0\) vanish.

Next let \(k>0\) be odd and \(l=1\).
 Then the dimension of \(\prod_{i\in J} M_i\) with \(\alpha_{k,l,J}\neq 0\) is divisible by four and the action on this product is of even type.

We will construct below two semi-free Spin \(S^1\)-manifolds \(N_1\) and \(N_2\)  with even action  from a semi-free Spin \(S^1\)-manifold \(M\) of dimension divisible by four with even action  such that
\(\lambda(N_1)=4 X_1X_0 (\lambda(M)-\bar{M})\) and \(\lambda(N_2)= 2 X_0^2 (\lambda(M)-\bar{M})\), where \(\bar{M}\) denotes the manifold \(M\) with trivial \(S^1\)-action.
\(N_1\) and \(N_2\) admit invariant metrics of positive scalar curvature.
Therefore after adding multiples of a manifold which is constructed by iterating these constructions we may assume that all \(\alpha_{k,l,J}\) with \(k>0\), \(l=1\) and non-empty \(J\) vanish.

For the construction of \(N_1\) consider the projectivization \(P\) of \(\C\oplus \gamma\otimes \gamma\), where \(\gamma\) is the dual of the Hopf bundle over \(\C P^1\).
\(S^1\) acts on \(P\) by multiplication on \(\gamma\otimes \gamma\).
Then \(P^{S^1}\) has two components and
\begin{equation*}
  \lambda(P)=2X_1-2X_1.
\end{equation*}

Let \(\gamma'\) be the dual of the tautological bundle over \(P\).
Then the \(S^1\)-action lifts into \(\gamma'\otimes \gamma'\) in such a way that the action over the two fixed point components in \(P\) are given by multiplication and multiplication with the inverse, respectively.
This action induces a semi-free \(S^1\)-action on \(P(\C\oplus \gamma'\otimes\gamma')\).
The fixed point set of the \(S^1\)-action on \(P(\C\oplus \gamma'\otimes\gamma')\) has four components and 
\begin{align*}
  \lambda(P(\C\oplus \gamma'\otimes \gamma'))&= 2X_1(X_0 -X_0)-2X_1(-X_0+X_0)\\ &= (2X_1X_0+2X_1X_0) - (2X_1X_0+ 2X_1X_0)
\end{align*}
Let \(E\) be the principal \(S^1\)-bundle associated to the dual of the tautological line bundle \(\gamma''\) over \(P(\C\oplus \gamma'\otimes \gamma')\).
The \(S^1\)-action on \(P(\C\oplus \gamma'\otimes \gamma')\) lifts into \(E\) in such a way that it is trivial over the two fixed point components corresponding to  \(-(2X_1X_0+ 2X_1X_0)\) and multiplication and multiplication with the inverse over the other two.
Therefore it follows that for \(N_1=E\times_{S^1}M\), \(\lambda(N_1)= 4X_1X_0 (M-\bar{M})\).
As in Construction~\ref{sec:non-spin-case-1} one sees that \(N_1\) admits an invariant metric of positive scalar curvature.
Because the action on \(M\) is even, \(N_1\) is spin.
The construction of \(N_2\) is similar with \(P\) replaced by \(\C P^1(\rho)\).
We omit the details.

By the above constructions we may now assume that
\begin{equation*}
  \lambda(M)=\sum_{k=1}^{(n-1)/2}X_1X_0^k\beta_k
\end{equation*}
with \(\beta_k\in \Omega_*^{\text{Spin}}\otimes \mathbb{Z}[\frac{1}{2}]\).
But the \(\mu(X_1X_0^k)\) are part of a basis of \(\Omega_*^{SO}(BU(1))\otimes\mathbb{Z}[\frac{1}{2}]\).
Therefore all the \(\beta_k\) must vanish and Theorem~\ref{sec:spin-case-1} is proved.
\end{proof}

\bibliography{circle_psc}{}
\bibliographystyle{amsplain}
\end{document}